\documentclass[12pt, reqno]{amsart}
\usepackage{amsmath, amssymb, amscd, amsthm, comment, amsxtra}
\usepackage{graphicx}
\usepackage{subfigure}
\usepackage{enumerate}
\usepackage[pdftex, linktocpage=true]{hyperref}
\usepackage{euscript}
\usepackage[format=plain,labelfont=bf,up,width=.99\textwidth]{caption}
\usepackage[toc,page]{appendix}
\usepackage{xcolor}
 \usepackage[colorinlistoftodos]{todonotes}
 \usepackage{marginnote}

\theoremstyle{plain}
\newtheorem{thm}{Theorem}[section]
\newtheorem{lem}[thm]{Lemma}
\newtheorem{prop}[thm]{Proposition}
\newtheorem{cor}[thm]{Corollary}
\newtheorem*{mainthm}{Main Theorem}

\theoremstyle{definition}
\newtheorem{defn}[thm]{Definition}

\newtheorem{rem}[thm]{Remark}

\numberwithin{equation}{section}
\newcommand{\thmref}[1]{Theorem~\ref{#1}}
\newcommand{\propref}[1]{Proposition~\ref{#1}}
\newcommand{\lemref}[1]{Lemma~\ref{#1}}
\newcommand{\corref}[1]{Corollary~\ref{#1}}
\newcommand{\figref}[1]{Figure~\ref{#1}}
\newcommand{\secref}[1]{Section~\ref{#1}}
\newcommand{\subsecref}[1]{Subsection~\ref{#1}}
\newcommand{\appref}[1]{Appendix~\ref{#1}}

\newcommand{\defnref}[1]{Definition~\ref{#1}}

\newcommand{\md}[1]{\;(\operatorname{mod}\; #1)}

\renewcommand{\epsilon}{\varepsilon}

\newcommand{\bbB}{\mathbb B}

\newcommand{\bbD}{\mathbb D}

\newcommand{\bbN}{\mathbb N}

\newcommand{\bbP}{\mathbb P}

\newcommand{\bbR}{\mathbb R}

\newcommand{\bbZ}{\mathbb Z}

\newcommand{\cB}{\mathcal B}
\newcommand{\cC}{\mathcal C}

\newcommand{\cE}{\mathcal E}
\newcommand{\cF}{\mathcal F}
\newcommand{\cG}{\mathcal G}

\newcommand{\cI}{\mathcal I}
\newcommand{\cJ}{\mathcal J}
\newcommand{\cK}{\mathcal K}
\newcommand{\cL}{\mathcal L}

\newcommand{\cP}{\mathcal P}

\newcommand{\cR}{\mathcal R}

\newcommand{\cT}{\mathcal T}
\newcommand{\cU}{\mathcal U}
\newcommand{\cV}{\mathcal V}

\newcommand{\cY}{\mathcal Y}
\newcommand{\cZ}{\mathcal Z}

\newcommand{\bfC}{\mathbf C}

\newcommand{\bfK}{\mathbf K}

\newcommand{\bfb}{\mathbf b}

\newcommand{\hB}{\hat B}

\newcommand{\hE}{\hat E}

\newcommand{\hH}{\hat H}

\newcommand{\ha}{\hat a}

\newcommand{\hk}{\hat k}

\newcommand{\hm}{\hat m}

\newcommand{\tiA}{\tilde A}
\newcommand{\tiB}{\tilde B}
\newcommand{\tiC}{\tilde C}

\newcommand{\tiE}{\tilde E}

\newcommand{\tiQ}{\tilde Q}

\newcommand{\tia}{\tilde a}

\newcommand{\tig}{\tilde g}

\newcommand{\bepsilon}{{\bar \epsilon}}

\newcommand{\bL}{{\bar L}}

\newcommand{\tixi}{\tilde \xi}

\newcommand{\tipsi}{\tilde \psi}

\newcommand{\tiGamma}{\tilde \Gamma}

\newcommand{\tiPhi}{\tilde \Phi}

\newcommand{\udelta}{\underline{\delta}}

\newcommand{\ueta}{\underline{\eta}}
\newcommand{\baeta}{\bar\eta}

\newcommand{\chcU}{\check{\cU}}

\newcommand{\chcI}{\check{\cI}}

\newcommand{\ticI}{\tilde{\cI}}
\newcommand{\hcI}{\hat{\cI}}

\newcommand{\hcJ}{\hat{\cJ}}

\newcommand{\ticJ}{\tilde{\cJ}}

\newcommand{\chl}{\check{l}}
\newcommand{\chH}{\check{H}}

\newcommand{\hcB}{\hat{\cB}}

\newcommand{\ticB}{\tilde\cB}

\newcommand{\frd}{{\frd}}

\newcommand{\ticE}{{\tilde\cE}}

\newcommand{\chA}{{\check A}}

\newcommand{\bC}{{\bar C}}

\newcommand{\Jac}{\operatorname{Jac}}

\newcommand{\dist}{\operatorname{dist}}

\newcommand{\Id}{\operatorname{Id}}
\newcommand{\loc}{\operatorname{loc}}

\newcommand{\Cr}{\operatorname{Cr}}
\newcommand{\CrD}{\operatorname{CrD}}

\newcommand{\Dis}{\operatorname{Dis}}

\newcommand{\Piod}{\Pi_{\operatorname{1D}}}

\newcommand{\matsp}[1]{\hspace{5mm} \text{#1} \hspace{5mm}}

\newcommand{\comma}{, \hspace{5mm}}

\title[A Priori Bounds for H\'enon-like Renormalization]{A Priori Bounds for H\'enon-like Renormalization}
\author{Sylvain Crovisier, Mikhail Lyubich, Enrique Pujals, Jonguk Yang}
\thanks{The first author was partially supported by the ANR project CoSyDy. The second author was partly supported by the NSF, the Hagler and Clay Fellowships, the Institute for Theoretical Studies at ETH (Zurich), SLMSI (formerly MSRI Berkeley), the Fields Institute and the Center of Nonlinear Analysis and Modeling at the University of Toronto. The third author was partly supported by NSF award number 24000048. The fourth author was partially supported by SLMSI, Institut Mittag-Leffler, the thematic programs `Topological, smooth and holomorphic dynamics, ergodic theory, fractals’ of the Simons Foundation Award No. 663281 for IMPAN, the Fields Institute and the Center of Nonlinear Analysis and Modeling at the University of Toronto.}

\begin{document}

\maketitle

\begin{abstract}
We formulate and prove \emph{a priori} bounds for the renormalization of H\'enon-like maps (under certain regularity assumptions). This provides a certain uniform control on the small-scale geometry of the dynamics, and ensures pre-compactness of the renormalization sequence. In a sequel to this paper \cite{Y}, a priori bounds are used in the proof of the main results, including renormalization convergence, finite-time checkability of the required regularity conditions and regular unicriticality of the dynamics. 
\end{abstract}

\tableofcontents


\section{Introduction}\label{sec.intro}

In dynamics, {\it a priori} bounds refer to uniform geometric control over a certain class of dynamical systems at all scales. This concept can also be understood as the precompactness of the renormalization orbits within that class in an appropriate topology. Establishing a priori bounds is often the most challenging step in the development of renormalization theory, as it is essential for achieving asymptotic universality of the dynamical systems on small scales.

In one-dimensional dynamics, both real and complex, there have been many occasions in the past forty years when a priori bounds were successfully established and led to desired conclusions. This includes the theory of circle diffeomorphisms (\cite{Her1}, \cite{Yo}, \cite{GoYa}), real one-dimensional unimodal maps (\cite{Gu}, \cite{Su}, \cite{LYam}), critical circle maps (\cite{Sw}, \cite{Her2}, \cite{DF}, \cite{Yam}), and complex quadratic-like maps (\cite{L3}, \cite{KaL}, \cite{DuL2}), Siegel maps (\cite{Mc}, \cite{InSh}, \cite{DuL1}).

A priori bounds serve as a crucial first step in proving renormalization  convergence by providing a uniform control over the size of the iterates of the map across renormalization steps. Without a priori bounds, the iterated maps could grow unboundedly or deviate significantly at each scale, making it impossible to apply the renormalization operator in a consistent way. These bounds thus ensure that each rescaled map remains within a confined range, preserving the conditions necessary for the renormalization process to reveal stable, convergent patterns and self-similarity across iterations.

 In higher-dimensional settings, obtaining such bounds is even more challenging. Consequently, developments in two-dimensional renormalization theories so far rely either on perturbative methods based on the robustness of one-dimensional theories (see \cite{DCLMa} and the references therein) or the aid of computer assistance (see \cite{GaJoMa} and the references therein).

In this series of papers, we set the stage for theoretical non-perturbative a priori bounds machinery for real two-dimensional H\'enon-like maps. The work begins with \cite{CLPY1}, where we introduce a class of regularly unicritical, infinitely renormalizable diffeomorphisms. For maps in this class, we show that the return maps to renormalization domains are regular H\'enon-like, meaning they have topological 2D horsehoe shapes that converge super-exponentially fast to 1D parabolic shapes. This property is the entrance point to the current paper, where we establish crucial a priori bounds for this class of maps that provide us with the bounded distortion and uniform differentiable shapes of all the renormalizations. In the forthcoming paper (\cite{Y}), these results will be used for establishing renormalization convergence and universality in the class of H\'enon-like maps. Furthermore, it will be shown that the aforementioned assumptions on the return maps are only required up to some finite scale, as a priori bounds ensure that they are automatically satisfied for any number of deeper renormalizations. This gives an opening for a rigorous search (with some computer assistance) of parameters in the H\'enon family that are infinitely renormalizable in the Feigenbaum universality class.


\subsection{Statement of the Main Theorem}

A {\it H\'enon-like map} $F : B \to F(B) \Subset B$ is a diffeomorphism defined on a rectangle $B = I\times J \subset \bbR^2$ of the form
\begin{equation}\label{eq.henonlike}
F(x,y) = (f(x,y), x)
\matsp{for}
(x,y) \in B,
\end{equation}
such that for any $y \in J$, the map $f(\cdot,y) : I \to I$ is a unimodal map. One may visualize the action of $F$ as bending $B$ into a U-shape, and then turning it on its side. See \figref{fig.henontrans}.  We refer to $\Piod(F)(x) := f(x, 0)$ as the {\it 1D profile of $F$}.

\begin{figure}[h]
\centering
\includegraphics[scale=0.12]{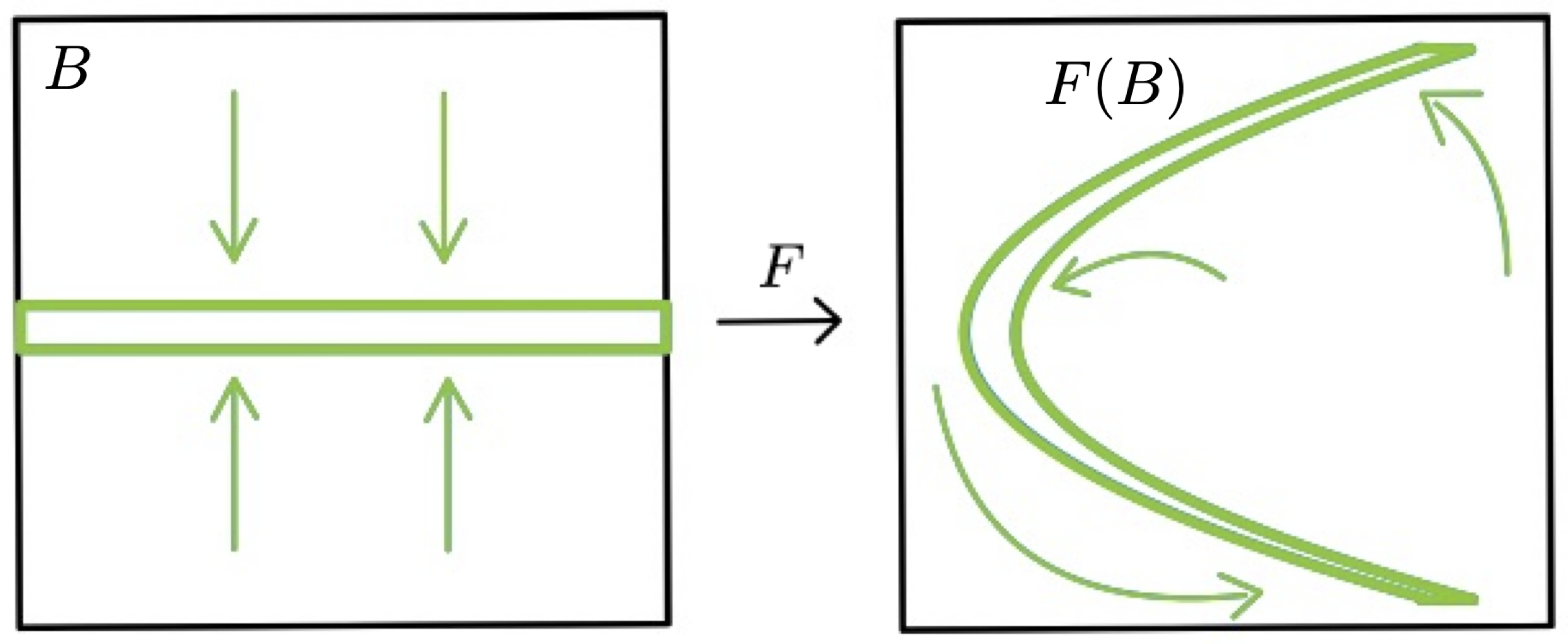}
\caption{H\'enon-like mapping.}
\label{fig.henontrans}
\end{figure}

We say that $F$ is {\it H\'enon-like renormalizable} if there is a subdomain $\cB^1 \Subset B$ that is {\it $R$-periodic}:
$$
F^i(\cB^1) \cap \cB^1 = \varnothing
\text{, when $1\leq i<R$ \quad and \quad }
F^R(\cB^1) \Subset \cB^1,
$$
and the return map $F^R|_{\cB^1}$ is again H\'enon-like after a smooth change-of-coordinates $\Phi : \cB^1 \to B^1$, referred to as a {\it straightening chart}, that preserve the genuine horizontal foliation. In this case, the pair $(F^R, \Phi)$ is referred to as a {\it H\'enon-like return}. We define the {\it (H\'enon-like) renormalization} $\cR(F)$ of $F$ as the H\'enon-like map obtained via a suitable affine rescaling of $\Phi\circ F^R \circ \Phi^{-1}$ that normalizes the width of the domain $B^1$. See \figref{fig.henonren}.

\begin{figure}[h]
\centering
\includegraphics[scale=0.06]{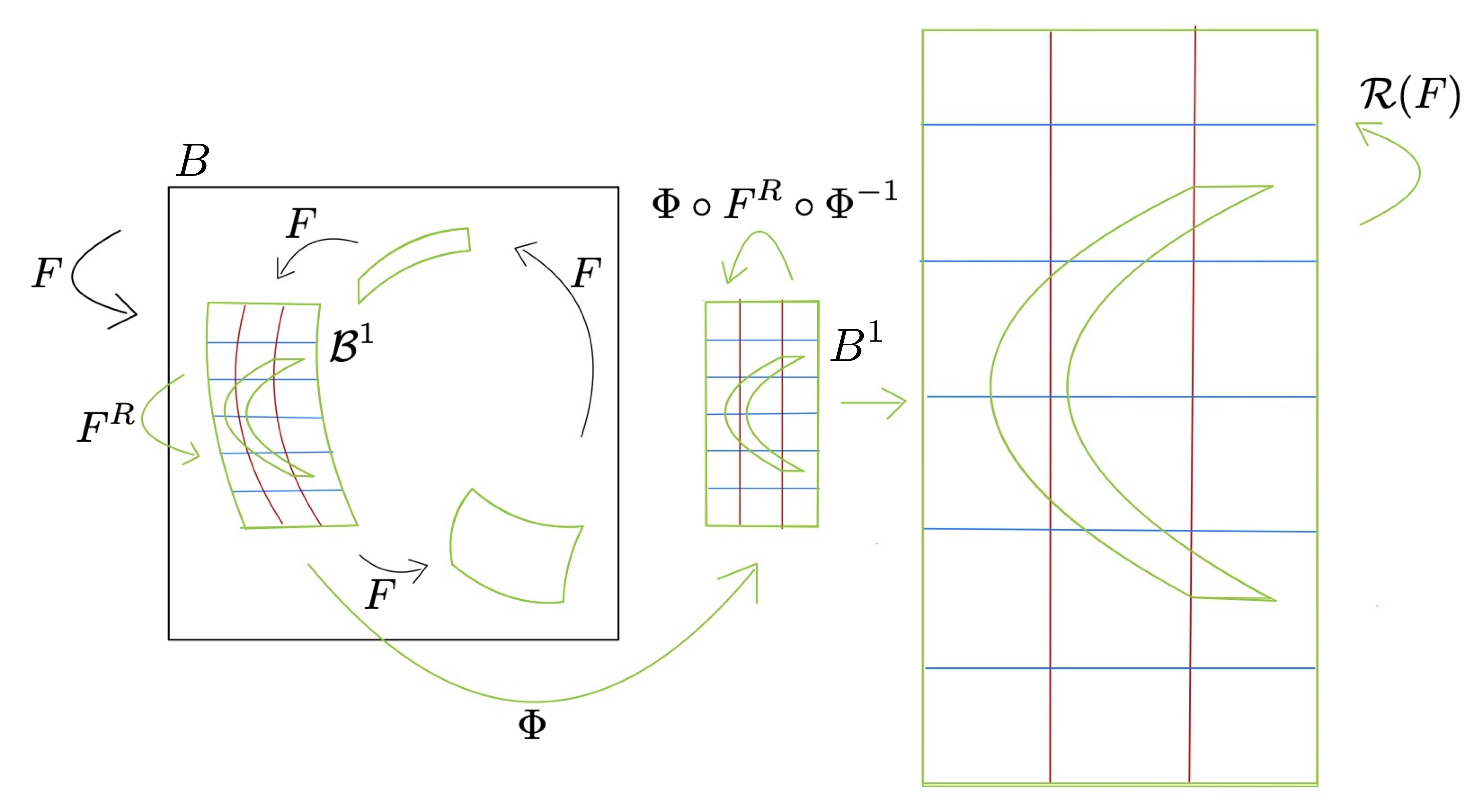}
\caption{H\'enon-like renormalization.}
\label{fig.henonren}
\end{figure}

As in the one dimensional setting, one is confronted with the problem that the geometry of the  new H\'enon-like map may not be  controlled. To address this issue, we incorporate Pesin theoretic ideas to the renormalization method. This involves keeping track of the {\it regularity} of points, which can then be used to control the geometry of dynamics in the higher-dimensional setting (see \appref{sec.pesin}). Loosely speaking, a H\'enon-like return $(F^R, \Phi : \cB^1 \to B^1)$ is {\it regular}, and $F$ is {\it regularly} H\'enon-like renormalizable if there is a sufficiently dominant exponential contraction of a vertical foliation of $\cB^1$ under $F^i$ for $1 \leq i \leq R$. For the precise definition, see \defnref{defn.reg henon return}.

\begin{rem}
As noted above, it is shown in \cite{CLPY1} that for a regularly unicritical infinitely renormalizable diffeomorphism, the return maps are eventually regular H\'enon-like.
\end{rem}

The H\'enon-like map $F : B \to B$ is $N$-times topologically renormalizable for some $N \in \bbN \cup \{\infty\}$ if there exist sequences
$$
B =: \cB^0 \Supset \cB^1\Supset \ldots
\matsp{and}
1 =: R_0 < R_1 < \ldots
$$
such that for $1 \leq n \leq N$, the set $\cB^n$ is an $R_n$-periodic Jordan domain. If there exists $\bfb \geq 2$ such that
$$
r_{n-1} := R_n/R_{n-1} \leq \bfb
\matsp{for all}
1\leq n \leq N,
$$
then we say that the combinatorics of renormalization for $F$ is {\it of ($\bfb$-)bounded type}. Suppose for $1 \leq n \leq N$, there exists a straightening chart $\Psi^n:\cB^n \to B^n$ such that $(F^{R_n}, \Psi^n)$ is a H\'enon-like return. Then the sequence
\begin{equation}\label{eq.returns}
\{(F^{R_n}, \Psi^n:\cB^n \to B^n)\}_{n=1}^N
\end{equation}
is said to be {\it nested}. Conjugating $\Psi^n \circ F^{R_n} \circ (\Psi^n)^{-1}$ by an affine map that normalizes the width of $B^n$, we obtain the $n$th renormalization $\cR^n(F)$ of $F$.

In the one dimensional context, a priori bound refers to a uniform control of the distortion of the derivatives along renormalization intervals. For the two dimensional context, we propose to consider the distortion of the derivative along some particular type of curves. More precisely,  consider a $C^1$-diffeomorphism $G : U \to G(U)$ defined on a domain $U \subset \bbR^2$. For a $C^1$-curve $\gamma \subset U$, let $\phi_\gamma : [0, |\gamma|] \to \gamma$ be the arc-length parameterization of $\gamma$. Denote $G_\gamma := \phi_{G(\gamma)}\circ G \circ \phi_\gamma^{-1}$. The {\it distortion of $G$ along $\gamma$} is defined as
$$
\Dis(G, \gamma) := \sup_{s, t \in (0, |\gamma|)}\frac{|G_\gamma'(s)|}{|G_\gamma'(t)|}.
$$

\begin{mainthm}
Consider a $C^6$-H\'enon-like map $F : B \to B$. Suppose for $N \in \bbN \cup \{\infty\}$, the map $F$ has $N$ nested regular H\'enon-like returns given by \eqref{eq.returns} with combinatorics of bounded type. For $1 \leq n \leq N$, let $\gamma^n$ be a genuine horizontal arc contained in $\cB^n$. Then $\Dis(F^{R_n}, \gamma^n)$ is uniformly bounded.
\end{mainthm}

\begin{rem} In the Main theorem, the uniform bound on $\Dis(F^{R_n}, \gamma^n)$ depends only on certain finite geometric data about the map and the parameters of regularity (listed in \eqref{eq.const 0} and \eqref{eq.const 1}). In particular, the bound does not depend on $n$ or $N$.
\end{rem}

Suppose that $N=\infty$. For $n \in \bbN$, consider the 1D profile $f_n := \Piod\circ\cR^n(F)$ of the $n$th renormalization of $F$. A consequence of the Main Theorem is that there exists a $C^2$-diffeomorphism $\phi_n$ and a constant $a_n \in \bbR$ such that the unimodal map $f_n$ decomposes into
$$
f_n(\cdot) = (\phi_n(\cdot))^2 + a_n.
$$
It follows that the sequence $\{f_n\}_{n=1}^\infty$ is pre-compact in the space of unimodal maps in the $C^1$-topology.

\subsection{Sketch of the Proof}

The proof of {\it a priori} bounds for 1D unimodal maps relies on the Koebe Distortion Principle that controls the ratio distortion for compositions of 1D diffeomorphisms. This principle combines the Denjoy Lemma (that controls the distortion away from the critical points) with the negative Schwarzian derivative property near the critical points (that gives control of the cross ratio distortion at these moments). In turn, the Denjoy Lemma fundamentally relies on the disjointness of the intervals in the cycle of a periodic subinterval, which ensures that their total length remains uniformly bounded.

Turning to the 2D case, consider the regular H\'enon-like returns given by \eqref{eq.returns}. For concreteness, assume that $N=\infty$. For $n \geq 0$, we analyze how much distortion is induced on a horizontal cross-section $\cI^n$ of the renormalization domain $\cB^n$ by the return map $F^{R_n}$. However, in this 2D setting, it is not obvious whether the 1D Koebe Distortion Principle can still be applied. For instance, a union of disjoint arcs in 2D space may have arbitrarily large total length. Moreover, since a H\'enon-like map is a diffeomorphism, it is not clear what, if anything, plays the role of its critical point. The following outlines our approach to overcoming these issues.

As our first task, we locate the {\it dynamical critical value} $v_0$ given by
$$
\{v_0\} = \bigcup_{n = 1}^\infty F^{R_n}(\cB^n).
$$
The point $v_0$ has a well-defined strong-stable manifold $W^{ss}(v_0)$ and center manifold $W^c(v_0)$ that have a quadratic tangency at $v_0$ (see \figref{fig.critval}). Using the quantitative estimates provided in \appref{sec.pesin}, we prove that the vertical foliations on $\cB^n$ induced by $\Psi^n$ converge super-exponentially fast to $W^{ss}(v_0)$. Moreover, the image $F^{R_n}(\gamma_h)$ of any horizontal arc $\gamma_h \subset \cB^n$ is super-exponentially close to a subarc of $W^c(v_0)$.

\begin{figure}[h]
\centering
\includegraphics[scale=0.07]{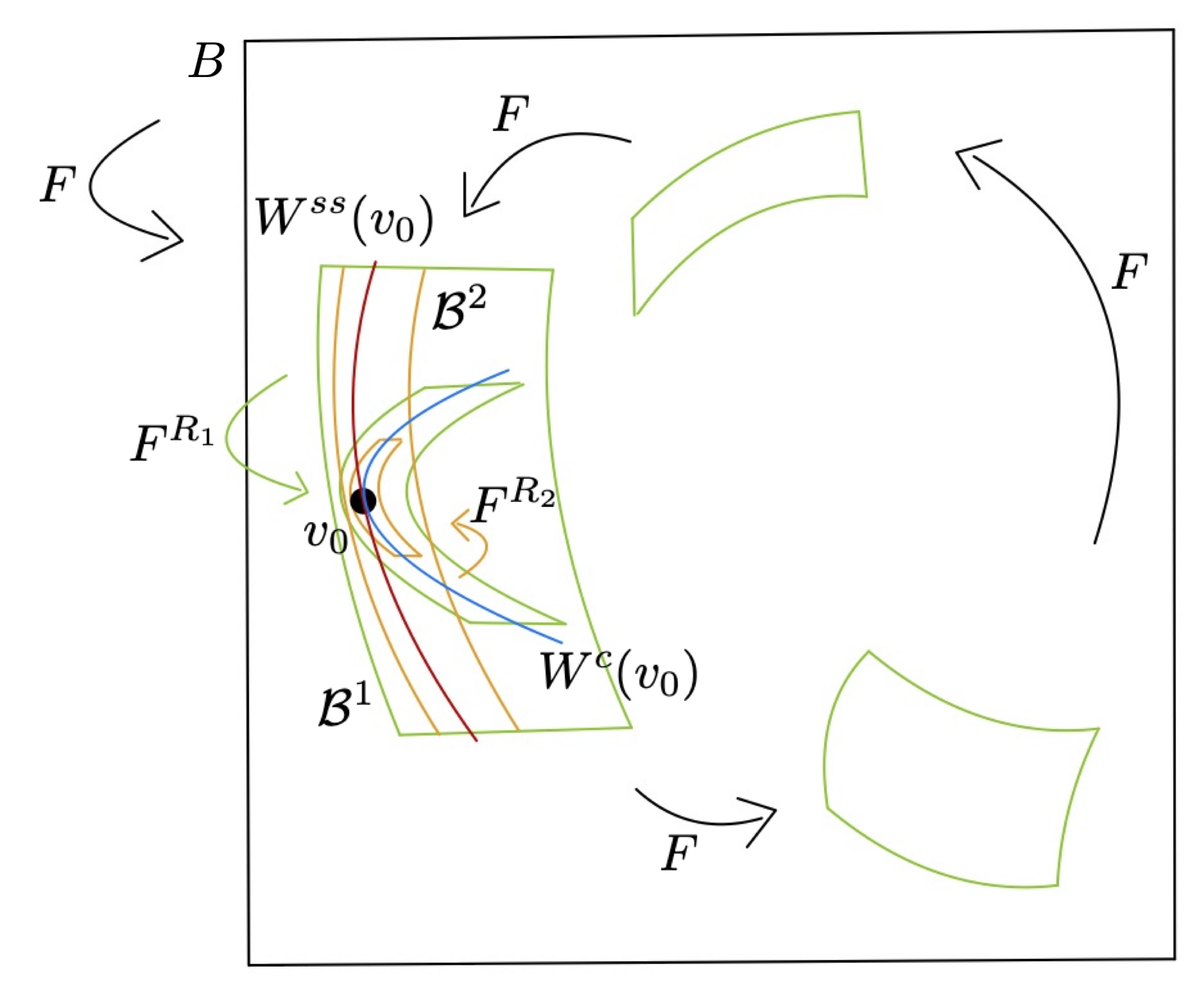}
\caption{The critical value $v_0$ of an infinitely regularly H\'enon-like renormalizable map $F$.}
\label{fig.critval}
\end{figure}

Assume that $\cI^m \supset \cI^n \ni v_0$ for $m \leq n$. For $n \geq n_0$ with $ n_0\geq 0$ sufficiently large, there exists a well-defined projection map $\cP^n : \cB^n \to \cI^n$ induced by the vertical foliation on $\cB^n$. Since the vertical foliation on $\cB^n$ converge super-exponentially to $W^{ss}(v_0)$, and $W^c(v_0)$ has a quadratic tangency with $W^{ss}(v_0)$, applying $\cP^n$ on a curve $\gamma$ sufficiently close to a subarc of $W^c(v_0)$ has the effect of applying a 1D quadratic map near its critical point.

By replacing $F$ with $F^{n_0}$, we can assume that $n_0 = 0$. Then by applying the projections $\cP^m$ for $0 \leq m < n$, we can confine the orbit of $\cI^n$ under $F$ to the fixed one-dimensional curve $\cI^0$. This reduces the 2D dynamics of $F$ on $B = \cB^0$ to the 1D mapping scheme on $\cI^0$ that gives the transitions from one projected iterated image of $\cI^n$ to another.

The main difficulty is that the foliation on $\cB^m$, through which $\cP^m$ is defined, is only invariant up to $R_m$ iterates. If $\cP^m$ is naively applied to the part of the orbit of $\cI^n$ contained in $\cB^m$, then the resulting projections to $\cI^0$ do not faithfully represent the original 2D dynamics of $F$ acting on $\cI^n$. For example, if $m \ll n$, then it is possible that the projections under $\cP^m$ of two iterated images of $\cI^n$ are not disjoint.

For this reason, the 1D reduction must be done in a more careful way. At various moments in the orbit of $\cI^n$, multiple projections from different scales of renormalization may be applied before they are ``undone'' to recover the original dynamics. In this procedure, we heavily rely on the fact that for bounded type combinatorics, the configuration of the cycle of $\cB^{m+s}$ for $s \in \{1, 2\}$ within $\cB^m$ resembles the 1D combinatorial structure (including some 2D version of the Denjoy disjointedness property). Then once the 2D dynamics of $F$ is properly reduced to a 1D mapping scheme, the Koebe Distortion Principle can be applied to obtain the desired bound on the distortion.

\subsection{Notations and conventions}

Unless otherwise specified, we adopt the following notations and conventions. 

Any diffeomorphism on a domain in $\bbR^2$ is assumed to be orientation-preserving. The projective tangent space at a point $p \in \bbR^2$ is denoted by $\bbP^2_p$.

Given a number $\eta > 0$, we use $\baeta$ to denote any number that satisfy
$$
\eta < \baeta < C\eta^D
$$
for some uniform constants $C > 1$ and $D > 1$ (if $\eta >1$) or $D \in (0, 1)$ (if $\eta < 1$) that are independent of the map being considered. Additionally, we allow $\baeta$ to absorb any uniformly bounded coefficient or power. So for example, if $\baeta >1$, then we may write
$$
\text{``}\;\; 10\baeta^5 = \baeta\;\;\text{''}.
$$
Similarly, we use $\ueta$ to denote any number that satisfy
$$
c\eta^d < \ueta < \eta
$$
for some uniform constants $c \in (0, 1)$ and $d \in (0, 1)$ (if $\eta >1$) or $d >1$ (if $\eta <1$) that are independent of the map being considered. As before, we allow $\ueta$ to absorb any uniformly bounded coefficient or power. So for example, if $\ueta >1$, then we may write
$$
\text{``}\;\;\tfrac{1}{3}\ueta^{1/4} = \ueta\;\;\text{''}.
$$
These notations apply to any positive real number: e.g. $\bepsilon > \epsilon$, $\udelta < \delta$, $\bL > L$, etc.

Note that $\baeta$ can be much larger than $\eta$ (similarly, $\ueta$ can be much smaller than $\eta$). Sometimes, we may avoid the $\baeta$ or $\ueta$ notation when indicating numbers that are somewhat or very close to the original value of $\eta$. For example, if $\eta \in (0,1)$ is a small number, then we may denote $\eta':=(1-\baeta)\eta$. Then $\ueta\ll \eta' < \eta$.

We use $n, m, i, j$ to denote integers (and less frequently $l, k$). The letter $i$ is never the imaginary number. Typically (but not always), $n \in\bbN$ and $m \in \bbZ$. We typically use $N, M$ to indicate fixed integers (often related to variables $n, m$).

We typically denote constants used for estimate bounds by $C, K \geq 1$ (less frequently $c > 0$).

We use calligraphic font $\cU, \cT, \cI,$ etc, for objects in the phase space and regular fonts $U, T, I,$ etc, for corresponding objects in the linearized/uniformized coordinates. A notable exception are for the invariant manifolds $W^{ss}, W^c$.

We use $p, q$ to indicate points in the phase space, and $z, w$ for points in linearized/uniformized coordinates.

For any set $X_m \subset \Omega$ with a numerical index $m \in \bbZ$, we denote
$$
X_{m+l} := F^l(X_m)
$$
for all $l \in \bbZ$ for which the right-hand side is well-defined. Similarly, for any direction $E_{p_m} \in \bbP^2_{p_m}$ at a point $p_m \in \Omega$, we denote
$$
E_{p_{m+l}} := DF^l(E_{p_m}).
$$

Define
$$
\pi_h(x,y) := x
\comma
\pi_v(x,y) := y
\comma
\Pi_h(x,y) := (x,0)
\matsp{and}
\Pi_v(x,y) := (0,y).
$$


 \section{Precise Definitions}\label{sec.defn}

\subsection{Charts}

For $z \in \bbR^2$, denote the space of tangent directions at $z$ by $\bbP^2_z$. Let $E^{gv}_z, E^{gh}_z\in \bbP^2_z$ be the genuine vertical and horizontal directions at $z$ respectively.

A {\it $C^r$-chart} is a $C^r$-diffeomorphism $\Phi : \cB \to B$ from a quadrilateral $\cB \subset \bbR^2$ to a rectangle $B = I \times J \subset \bbR^2$, where $I, J \subset \bbR$ are intervals. The {\it vertical/horizontal direction $E^{v/h}_p \in \bbP^2_p$ at $p\in\cB$ (associated to $\Phi$)} are given by
$$
E^{v/h}_p := D\Phi^{-1}\left(E^{gv/gh}_{\Phi(p)}\right).
$$
The chart $\Phi$ is said to be {\it genuinely vertical/horizontal} if $E^{v/h}_p = E^{gv/gh}_p$ for all $p \in\cB$.

A {\it vertical leaf in $\cB$} is a curve $l^v$ such that
$$
l^v \subseteq \Phi^{-1}(\{a\} \times \pi_v(B))
\matsp{for some}
a \in \pi_h(B).
$$
If the above containment is an equality, then $l^v$ is said to be {\it full}. A {\it (full) horizontal leaf $l^h$ in $\cB$} is defined analogously.

Let $p \in \cB$ and $E_p \in \bbP^2_p$. Denote
$
z := \Phi(p)
$ and $
E_z := D\Phi(E_p).
$
For $t > 0$, the direction $E_p$ is said to be {\it $t$-vertical in $\cB$} if 
$$
\frac{\measuredangle(E_z, E_z^{gv})}{\measuredangle(E_z, E_z^{gh})} < t.
$$
A {\it $t$-horizontal direction in $\cB$} is analogously defined.

A $C^0$-curve $\Gamma^v \subset \cB$ is said to be {\it vertical in $\cB$} if $\Phi(\Gamma^v)$ is a vertical graph in $B$ in the usual sense. That is, there exists an interval $I^v \subseteq \pi_v(B)$ and a map $g_v : I^v \to \pi_h(B)$ such that
$$
\Phi(\Gamma^v) = \cG^v(g_v) := \{(g_v(y), y) \; | \; y \in I^v\}.
$$
If $I^v = \pi_v(B)$, then $\Gamma^v$ is said to be {\it vertically proper in $\cB$}. A {\it horizontal} or a {\it horizontally proper curve $\Gamma^h$ in $\cB$} is analogously defined. If $\Gamma^v$ is $C^r$, and $\|g_v'\|_{C^{r-1}} \leq t$ for some $t \geq 0$, then we say that $\Gamma^v$ is {\it $t$-vertical (in $C^r$) in $\cB$}. Note that $\Gamma^v$ is a (vertically proper) $0$-vertical curve if and only if it is a (full) vertical leaf.

If $\Gamma^v$ is $C^2$, and $g_v$ has a unique critical point $c \in I^v$ of quadratic type: $g_v'(c) =0$ and
\begin{equation}\label{eq.val curv}
\kappa_\Phi(\Gamma^v) := g_v''(c) \neq 0,
\end{equation}
then $\Gamma^v$ is a {\it vertical quadratic curve in $\cB$}. We refer to $\kappa_\Phi(\Gamma^v)$ as the {\it valuable curvature of $\Gamma^v$ in $\cB$}.

Let $\cE^v : \cB \to T^1(\cB)$ be the $C^{r-1}$-unit vector field given by
$$
\cE^v(p) := D\Phi^{-1}(E^{gv}_{\Phi(p)}).
$$
A $C^{r-1}$-unit vector field $\ticE^v : \cU \to T^1(\cU)$ defined on a domain $\cU \subset \cB$ is said to be {\it $t$-vertical in $C^{r-1}$ in $\cB$} for some $t \geq 0$ if $\|\ticE^v - \cE^v\|_{C^{r-1}} \leq t$.

Let $\tiPhi : \ticB \to \tiB$ be another chart with $\ticB \subset \cB$. We define the following relations between $\Phi$ and $\tiPhi$.
\begin{itemize}
\item We say that $\ticB$ is {\it vertically proper in $\cB$} if every full vertical leaf in $\ticB$ is vertically proper in $\cB$.
\item We say that $\Phi$ and $\tiPhi$ are {\it horizontally equivalent on $\ticB$} if every horizontal leaf in $\ticB$ is a horizontal leaf in $\cB$.
\item For $t \geq 0$, we say that {\it $\ticB$ is $t$-vertical in $\cB$} if $\Phi$ and $\tiPhi$ are horizontally equivalent, and the unit vector field given by
$$
\ticE^v(p) := D\tiPhi^{-1}(E^{gv}_{\tiPhi(p)})
\matsp{for}
p \in \ticB
$$
is $t$-vertical in $C^{r-1}$ in $\cB$.
\item We say that $\Phi$ and $\tiPhi$ are {\it equivalent on $\ticB$} if $\ticB$ is $0$-vertical in $\cB$.
\end{itemize}

Let $\Psi : \cB \to B$ be a chart satisfying the following properties.
\begin{itemize}
\item There exists $q \in \cB$ such that $\Psi(q) = 0 \in B$.
\item Let
$$
\cI^h(t) := \Psi^{-1}(t, 0)
\matsp{for}
t \in \pi_h(B),
$$
and
$$
\cI^v(s) := \Psi^{-1}(0, s)
\matsp{for}
s \in \pi_v(B).
$$
Then $\|(\cI^{h/v})'\| \equiv 1$.
\end{itemize}
In this case, we say that $\Psi$ is {\it centered (at $q$)}. Clearly, for any chart $\Phi : \cB \to D$ and any point $q \in \cB$, there exists a unique chart $\Psi :\cB \to B$ equivalent to $\Phi$ that is centered at $q$.

Suppose that $\Psi : (\cB, q) \to (B, 0)$ is centered at some point $q \in \cB$. Let $\Gamma^h \subset \cB$ be a horizontal $C^r$-curve, so that $\Psi(\Gamma^h)$ is the horizontal graph in $B$ of a $C^r$-map $g_h : I^h \to \pi_v(B)$ defined on an interval $I^h \subset \pi_h(B)$. We say that $\Gamma^h$ is {\it $t$-horizontal in $C^r$ in $\cB$} if $\|g_h\|_{C^r} \leq t$. In particular, $\Gamma^h$ is $0$-horizontal in $\cB$ if and only if $\Gamma^h$ is a subarc of the full horizontal leaf containing $q$.

\subsection{Regularity}

For an integer $r \geq 2$, consider a $C^r$-diffeomorphism $F : B \to F(B) \Subset B$ defined on a domain $B\subset \bbR^2$. Let $L \geq 1$; $\epsilon, \lambda \in (0,1)$  and $M \in \bbN \cup \{\infty\}$. A point $p \in B$ is {\it $M$-times forward $(L, \epsilon, \lambda)$-regular along $E_p^+ \in \bbP^2_p$} if for $s \in \{0, 1\}$,  we have
\begin{equation}\label{eq.for reg}
L^{-1}\lambda^{(1+\epsilon)m}\leq \frac{\|DF^m|_{E^+_p}\|^{s+1}}{(\Jac_pF^m)^s} \leq L \lambda^{(1-\epsilon)m}
\matsp{for all}
1 \leq m \leq M.
\end{equation}
Intuitively, when $s = 0$, the condition gives exponential contraction along $E_p$, and when $s=1$, it gives the domination of this contraction (there is exponential repulsion in the projective tangent space of other directions from $E_p$). 
Similarly, $p$ is {\it $M$-times backward $(L, \epsilon, \lambda)$-regular along $E^-_p \in \bbP^2_p$} if for $s \in \{0, 1\}$,
\begin{equation}\label{eq.back reg}
L^{-1}\lambda^{(1+\epsilon)m}\leq \frac{(\Jac_pF^{-m})^s}{\|DF^{-m}|_{E^-_p}\|^{s+1}} \leq L \lambda^{(1-\epsilon)m}
\matsp{for all}
1 \leq m \leq M.
\end{equation}
The constants $L$, $\epsilon$ and $\lambda$ are referred to as an {\it irregularity factor}, a {\it marginal exponent} and a {\it contraction base} respectively.

If $\epsilon$ is sufficiently small (independently of $F$), then the local dynamics of $F$ near the forward (or backward resp.) orbit of $p$ can be quasi-linearized up to the $M$th iterate, see \thmref{reg chart}. If $M = \infty$, this implies in particular that $p$ has a $C^r$-smooth strong-stable manifold (or center manifold resp.), see \thmref{stable}. It should be noted that the center manifold at an infinitely backward regular point $p$ is not uniquely defined; however, its $C^r$-jet at $p$ is unique, see \thmref{center jet}. Henceforth, any marginal exponent will be assumed to be sufficiently small.

\begin{defn}\label{defn.reg henon return}
Let $F : B \to B$ be a H\'enon-like map. A H\'enon-like return $(F^R, \Phi : \cB^1 \to B^1)$ is said to be {\it $(L, \epsilon, \lambda)$-regular} if the following conditions hold. Let
$$
E_p^{v/h} := D\Psi^{-1}\left(E^{gv/gh}_{\Psi(p)}\right).
$$
\begin{itemize}
\item Every $p \in \cB^1$ is $R$-times forward $(L, \epsilon, \lambda)$-regular along $E^v_p$.
\item Every $q \in F^R(\cB^1) \Subset \cB^1$ is $R$-times backward $(L, \epsilon, \lambda)$-regular along $E^h_q$.
\item For any $p \in \cB^1$, we have $\measuredangle(E^v_p, E^h_p) > 1/L$.
\end{itemize}
In this case, we say that $F$ is {\it $(L, \epsilon, \lambda)$-regularly H\'enon-like renormalizable}.
\end{defn}


 \section{Convergence of the Straightening Charts}\label{sec.conv chart}
 
Let $r \geq 2$ be an integer, and consider a $C^{r+1}$-H\'enon-like map $F : B \to B$. For some $N \in \bbN \cup \{\infty\}$; $L \geq 1$ and $\epsilon, \lambda \in (0,1)$, suppose that $F$ has $N$ nested $(L, \epsilon, \lambda)$-regular H\'enon-like returns given by \eqref{eq.returns}. Furthermore, assume that $N$ is sufficiently large, so that for some smallest number $0 \leq n_0 \leq N$, we have
\begin{equation}\label{eq.proper depth 0}
\overline{K_0} \lambda^{\epsilon R_{n_0}} < \eta,
\end{equation}
where $\eta \in (0,1)$ is independent of $F$, and
\begin{equation}\label{eq.const 0}
{K_0} = {K_0}(L, \lambda, \epsilon, \lambda^{1-\epsilon}\|DF^{-1}\|, \|DF\|_{C^r}, r)\geq 1
\end{equation}
is a uniform constant.

For $n_0 \leq n \leq N$ and $m \in \bbZ$, denote $\cB^n_m := F^m(\cB^n)$. Observe that
$
\cB^{n+1}_{R_{n+1}} \Subset \cB^n_{R_n}.
$
Let
$$
v_0 \in \cZ_0 := \bigcap_{n=1}^N \cB^n_{R_n},
$$
be a point to be specified later (as the {\it critical value of $F$}). Without loss of generality, assume that $\Psi^n$ is centered at $v_0$.

In this section, we describe the asymptotic behavior of the centered straightening charts $\{\Psi^n\}_{n=1}^N$ for the renormalizations of $F$.

\begin{rem}
In this section, we do not assume that the combinatorics of the renormalizations of $F$ is necessarily of bounded type.
\end{rem}

Define
$$
I^n_0 := \pi_h(B^n_0)
\matsp{and}
\cI^n_0 := (\Psi^n)^{-1}(I^n_0 \times \{0\}).
$$
Then it follows that
$
I^n_0 \Subset I^1_0
$ and $
\Psi^n|_{\cI^n_0} =\Psi^1|_{\cI^n_0}.
$
Denote $\cI^n_m := F^m(\cI^n_0)$ for $m \in \bbZ$. 

For $p_0 \in \cB^n_0$, write $z_0 := \Psi^n(p_0)$, and let
$$
E^h_{p_0} := D(\Psi^n)^{-1}(E^{gh}_{z_0})
\matsp{and}
E^{v,n}_{p_0} := D(\Psi^n)^{-1}(E^{gv}_{z_0}).
$$
Additionally, let
$$
E^{h,n}_{p_{R_n-1}} := DF^{R_n-1}(E^h_{p_0})
\matsp{and}
E^v_{p_{R_n-1}} := DF^{-1}(E^h_{p_{R_n}}) = DF^{R_n-1}(E^{v,n}_{p_0}).
$$
By increasing $L$ by a uniform amount if necessary (see \propref{grow irreg}), we may assume that every $q \in \cB^n_{R_n-1}$ is $(R_n-1)$-times backward $(L, \epsilon, \lambda)$-regular along $E^v_q$.

\begin{prop}[Vertical extension of charts]\label{vert prop nest}
For $n_0 \leq n \leq N$, the $n$th centered straightening chart can be extended to $\Psi^n : \hcB^n_0 \to \hB^n_0$ such that the following properties hold.
\begin{enumerate}[i)]
\item The quadrilateral $\hcB^n_0$ is vertically proper and $\eta$-vertical in $\cB^{n_0}_0$.
\item We have $\|(\Psi^n)^{\pm 1}\|_{C^r} < {K_0}$, and
\begin{equation}\label{eq.psi conv}
\|\Psi^n\circ (\Psi^{n+1})^{-1}-\Id\|_{C^r} <\lambda^{(1-\bepsilon)R_n}.
\end{equation}
\item Every point $q_0 \in \hcB^n_0$ is $R_n$-times forward $({K_0}, \epsilon, \lambda)$-regular along $E^{v,n}_{q_0}$.
\end{enumerate}
\end{prop}

\begin{proof}
For $p_0 \in \cB^n_0$, let
$$
\{\Phi_{p_m} : \cU_{p_m} \to U_{p_m}\}_{m=0}^{R_n}
$$
be a linearization of $F$ along the $R_n$ forward orbit of $p_0$ with vertical direction $E^{v,n}_{p_0}$. Let
$
\cE^{v,n}_{p_m} : \cU_{p_m} \to T^1(\cU_{p_m})
$
be the $C^r$-unit vector field given by
$
\cE^{v,n}_{p_m}(q) \in E^{v,n}_q
$ for $
q \in \cU_{p_m}.
$

Let $l^{v, {n_0}}_{p_0}$ be the full vertical leaf in $\cB^{n_0}_0$ containing $p_0$. For $q_0 \in l^{v, n_0}_{p_0}$, let
$$
\{\Phi_{q_m} : \cU_{q_m} \to U_{q_m}\}_{m=0}^{R_{n_0}}
$$
be a linearization of $F$ along the $R_{n_0}$ forward orbit of $q_0$ with vertical direction $E^{v, {n_0}}_{q_0}$ given by \thmref{reg chart}.

Let $M$ be a nearest integer to $R_{n_0}/2$. By \lemref{size reg nbh}, we see that $\cU_{p_M} \supset \bbD_{p_M}\left(\lambda^{\bepsilon M}\right).$ \lemref{trunc neigh fit} implies that \corref{ver hor cons} applies to all points in the $M$-times truncated regular neighborhood $\cU_{q_0}^M$ at $q_0$. The $R_{n_0}$-times forward regularity at all points in $\cB^{n_0}_0$ together with \eqref{eq.proper depth 0} implies that
$$
\chcU_{q_M} := F^M(\cU_{q_0}^M) \subset \cU_{p_M}.
$$

By \propref{grow irreg}, $q_M$ and $p_M$ are $M$-times forward $(\lambda^{-\bepsilon M}, \epsilon, \lambda)_v$-regular along $E^{v,{n_0}}_{q_M}$ and $E^{v,n}_{p_M}$ respectively. Hence, \propref{vert angle shrink} implies that $\cE^{v,n}_{p_M}|_{\chcU_{q_M}}$ is $\lambda^{(1-\bepsilon)M}$-vertical in $C^0$ in $\cU_{q_M}$. Moreover, the bounds on $\|\Phi_{p_M}\|_{C^r}$ and $\|\Phi_{q_M}\|_{C^r}$ given by \thmref{reg chart} imply that
$$
\|D\Phi_{q_M}(\cE^{v,n}_{p_M})\|_{C^{r-1}} < \lambda^{-\bepsilon M}.
$$
Extend $\cE^{v,n}_{p_0}$ to $\cU_{q_0}^M$ as
$$
\cE^{v,n}_{p_0}|_{\cU_{q_0}^{M}} := DF^{-M}(\cE^{v,n}_{p_M}|_{\chcU_{q_M}}).
$$
Then by \propref{back dt}, we have
$$
\|\cE^{v,n}_{p_0}- \cE^{v,{n_0}}_{q_0}\|_{C^r} \leq \lambda^{(1-\bepsilon) M}(1+\|D\Phi_{q_0}^{-1}\|_{C^{r-1}})(1+\|D\Phi_{q_M}(\cE^{v,n}_{p_M})\|_{C^{r-1}}) \leq \eta.
$$
Rectifying the vertical directions near $l^{v, {n_0}}_{p_0}$ given by $\cE^{v,n}_{p_0}$, we obtain the desired extension of $\Psi^n$.

Observe that for $0 \leq k \leq M$
$$
\measuredangle(E^{v,n}_{q_k}, E^{v, n_0}_{q_k}) < \lambda^{(1-\bepsilon)(R_{n_0}-k)}. 
$$
It follows that
$$
\frac{1}{\sqrt{2}}< \frac{\|DF^k|_{E^{v,n}_{q_0}}\|}{\|DF^k|_{E^{v,{n_0}}_{q_0}}\|} < \sqrt{2}.
$$
Concatenating with the forward $M$-times $(\lambda^{-\bepsilon M}, \epsilon, \lambda)$-regularity of $q_M$, we see that
$$
\frac{\lambda^{\bepsilon M}}{\sqrt{2}L}\lambda^{(1+\epsilon)(M+i)} \leq \|DF^{M+i}|_{E^{v,n}_{q_0}}\| \leq \sqrt{2} L \lambda^{-\bepsilon M}\lambda^{(1-\epsilon)(M+i)}.
$$
Since $n_0$ is assumed to be the smallest number that satisfies \eqref{eq.proper depth 0}, we have $\lambda^{-\bepsilon M} < K_0$. The claimed $R_n$-times forward regularity of $q_0$ along $E^{v,n}_{q_k}$ follows.

Lastly, replacing the renormalization depth ${n_0}$ in the above argument by $n$, we obtain \eqref{eq.psi conv}.
\end{proof}

Consider $C^r$-curves $\Gamma_1, \Gamma_2 \subset \bbR^2$ with $|\Gamma_1| \geq |\Gamma_2|$. For $i \in \{1, 2\}$, let $\phi_{\Gamma_i} : J_i \subset \bbR \to \Gamma_i$ be a parameterization of $\Gamma_i$ such that
\begin{itemize}
\item $|\phi_{\Gamma_i}'| \equiv 1$;
\item $J_1 \supset J_2$;
\item $\|\phi_{\Gamma_1}|_{J_2} - \phi_{\Gamma_2}\|_{C^r}$ is minimal.
\end{itemize}
In this case, define
\begin{equation}\label{eq.dist curve}
\|\Gamma_1\|_{C^r} := \|\phi_{\Gamma_1}\|
\matsp{and}
\dist_{C^r}(\Gamma_1, \Gamma_2) := \|\phi_{\Gamma_1}|_{J_2} - \phi_{\Gamma_2}\|_{C^r}.
\end{equation}

\begin{lem}\label{hor curvs conv}
For $n_0 \leq n \leq N$, let $l^n_0$ be a full horizontal leaf in $\hcB^n_0$. Denote $l^n_m := F^m(l^n_0)$ for $m\in\bbZ$. Then we have $\|l^n_{R_n-1}\|_{C^r} < {K_0}$; and
$$
\dist_{C^r}(l^n_{R_n-1}, l^{n+1}_{R_{n+1}-1}) < \lambda^{(1-\bepsilon)R_n}.
$$
\end{lem}

\begin{proof}
For $p_{-1} \in \cZ_{-1} := F^{-1}(\cZ_0),$
let
$$
\{\Phi_{p_{-m}} : \cU_{p_{-m}}\to U_{p_{-m}}\}_{m=1}^{R_N}
$$
be a linearization of $F$ along the $R_N$-times backward orbit of $p_{-1}$ with vertical direction $E^v_{p_{-1}}$ (if $N= \infty$, then $R_\infty = \infty$). Let $\cV_{-R_n}$ be the connected component of
$
F^{-R_n+1}(\cU_{p_{-1}}^{R_n}) \cap \hcB^n_0
$
containing $p_{-R_n}$. Note that $\Psi^n|_{\cV_{-R_n}}$ defines a chart on $\cV_{-R_n}$, so that $\cV_{-R_n}$ is $0$-vertical in $\hcB^n_0$. Moreover, arguing as in the proof of \propref{vert prop nest}, we see that $\cV_{-R_n}$ is also vertically proper in $\hcB^n_0$. 

Consider the map
$$
H_n(x,y) = (h_n(x), e_n(x,y)) := \Phi_{p_{-1}}\circ F^{R_n-1}\circ (\Psi^n)^{-1}(x,y)
$$
for $(x,y)\in \Psi^n(\cV_{-R_n})$. Denote
$$
F_{p_{-n}} := \Phi_{p_{-n+1}}\circ F \circ (\Phi_{p_{-n}})^{-1}.
$$
Then
\begin{equation}\label{eq.decomp proof}
H_n = F_{p_{-2}} \circ \ldots \circ F_{p_{-R_n}} \circ \Phi_{p_{-R_n}}\circ (\Psi^n)^{-1}.
\end{equation}
By \thmref{reg chart}, we see that
$$
\|\Phi_{p_{-R_n}}\circ (\Psi^n)^{-1}\|_{C^r} < \lambda^{-\bepsilon R_n}.
$$
Applying \propref{cr lin comp}, we conclude that
\begin{equation}\label{eq.en bound}
\|e_n\|_{C^r} < \lambda^{(1-\bepsilon)R_n}.
\end{equation}
The result follows.
\end{proof}

\begin{prop}[Locating the critical value]\label{crit value}
If $N = \infty$, then the following statements hold.
\begin{enumerate}[i)]
\item For any point $p_0 \in \cZ_0$, there exists a unique strong stable direction $E^{ss}_{p_0} \in \bbP^2_{p_0}$ such that
$$
\|E^{v, n}_{p_0}-E^{ss}_{p_0}\| < \lambda^{(1-\bepsilon)R_n}
\matsp{for}
n \geq n_0.
$$
Moreover, $p_0$ is infinitely forward $(L, \epsilon, \lambda)$-regular along $E^{ss}_{p_0}$.
\item Any point $p_{-1} \in \cZ_{-1} := F^{-1}(\cZ_0)$ is infinitely backward $(L, \epsilon, \lambda)$-regular along $E^v_{p_{-1}}$. Moreover, there exists a unique center direction $E^c_{p_{-1}} \in \bbP^2_{p_{-1}}$ such that
$$
\|E^{h, n}_{p_{-1}}-E^c_{p_{-1}}\| < \lambda^{(1-\bepsilon)R_n}
\matsp{for}
n \geq n_0.
$$
\item There exists a unique point $v_0 \in \cZ_0$ such that
$$
E^{ss}_{v_0} = DF(E^c_{v_{-1}}).
$$
Moreover, the strong stable manifold $W^{ss}(v_0)$ and the center manifold $F(W^c(v_{-1}))$ have a quadratic tangency at $v_0$.
\end{enumerate}
\end{prop}

\begin{proof}
The first and second claim follow immediately from Propositions \ref{vert angle shrink} and \ref{hor angle shrink}.

For $n  \geq n_0$, recall that $\cI^n_{R_n}$ is a vertical quadratic curve in $\cB^n_0$. Let $v^n_0\in \cI^n_0$ be the unique point such that
$$
E^{v,n}_{v^n_{R_n}} = DF^{R_n}(E^h_{v^n_0}).
$$
By \propref{vert prop nest} and \lemref{hor curvs conv}, we have
$$
\dist(v^n_{R_n},v^{n+1}_{R_{n+1}}) < \lambda^{(1-\bepsilon)R_n}.
$$
Thus, there exists a unique point $v_0 \in \cZ_0$ such that
$$
\dist(v^n_{R_n}, v_0) < \lambda^{(1-\bepsilon)R_n}
\matsp{and}
\dist_{C^r}(\cI^n_{R_n}, W^c(v_0))  < \lambda^{(1-\bepsilon)R_n}.
$$
By \eqref{eq.psi conv}, we see that $W^{ss}(v_0)$ and $W^c(v_0)$ have a quadratic tangency at $v_0$.

Lastly, let $\cU_{v_0}$ be a neighborhood of $v_0$. Then there exists a uniform constant $k > 0$ such that for all $n$ sufficiently large, if $p_{R_n} \in \cI^n_{R_n} \setminus \cU_{v_0}$ then
$$
\measuredangle(E^{v,n}_{p_{R_n}}, DF^{R_n}(E^h_{p_0})) > k.
$$
Thus, $v_0$ is the unique point in $\cZ_0$ satisfying $E^{ss}_{v_0} = E^c_{v_0}.$
\end{proof}

We define the {\it critical value $v_0 \in \cZ_0$} as follows. If $N=\infty$, let $v_0$ be the point given in \propref{crit value} iii). Otherwise, let $v_0$ be the unique point in $\cI^N_{R_N}$ such that
$$
DF^{R_N}(E^h_{v_{-R_N}}) = E^{v,N}_{v_0}
$$
(recall that such a point exists since $\cI^N_{R_N}$ is a vertical quadratic curve in $\cB^N_0$). Define the {\it critical point} as $v_{-1} := F^{-1}(v_0)$.

\begin{thm}[Valuable charts]\label{crit chart}
Let ${K_0} \geq 1$ be the constant given in \eqref{eq.const 0}. There exist charts
$$
\Phi_0 : (\cB_0, v_0) \to (B_0, 0)
\matsp{and}
\Phi_{-1} : (\cB_{-1}, v_{-1}) \to (B_{-1}, 0)
$$
such that
\begin{itemize}
\item $\Phi_0$ is centered at $v_0$ and is genuine horizontal;
\item $\cB_0 \supset \cB^{n_0}_0$, $\cB_{-1} \supset \cB^{n_0}_{R_{n_0}-1}$ and $F(\cB_{-1}) \Subset \cB_0$;
\item $\|\Phi_i^{\pm 1}\|_{C^r} < {K_0}$ for $i \in \{0, -1\}$; and
\item we have
\begin{equation}\label{eq.henon trans}
\Phi_0 \circ F \circ \Phi_{-1}^{-1}(x,y) = (f_0(x) - \lambda y, x)
\matsp{for}
(x,y) \in B_{-1},
\end{equation}
where $f_0 : (\pi_h(B_{-1}), 0) \to (\pi_h(B_0), 0)$ is a $C^r$-map that has a unique critical point at $0$ such that
\begin{equation}\label{eq.second}
\|f_0''\|_{C^{r-2}} < {K_0}
\matsp{and}
\kappa_F:=\inf_{x\in\pi_h(B_{-1})} f_0''(x) > 0.
\end{equation} 
\end{itemize}
Moreover, the following properties hold for $n_0 \leq n \leq N$.
\begin{enumerate}[i)]
\item We have
$$
\|\Phi_0 \circ (\Psi^n)^{-1} - \Id\|_{C^r} <\lambda^{(1-\bepsilon)R_n}.
$$
\item Let
$$
H_n := \Phi_{-1}\circ F^{R_n-1}\circ (\Psi^n)^{-1}.
$$
Then $H_n(x,y) = (h_n(x), e_n(x,y))$, where $h_n : I^n_0 \to h_n(I^n_0)$ is a $C^r$-diffeomorphism and $e_n$ is a $C^r$-map such that
\begin{equation}\label{eq.first entry}
\lambda^{\bepsilon R_n} < |h_n'(x)| < \lambda^{-\bepsilon R_n}
\matsp{for}
x\in I^n_0
\matsp{and}
\|e_n\|_{C^r} <\lambda^{(1-\bepsilon)R_n}.
\end{equation}
\end{enumerate}
\end{thm}

\begin{figure}[h]
\centering
\includegraphics[scale=0.15]{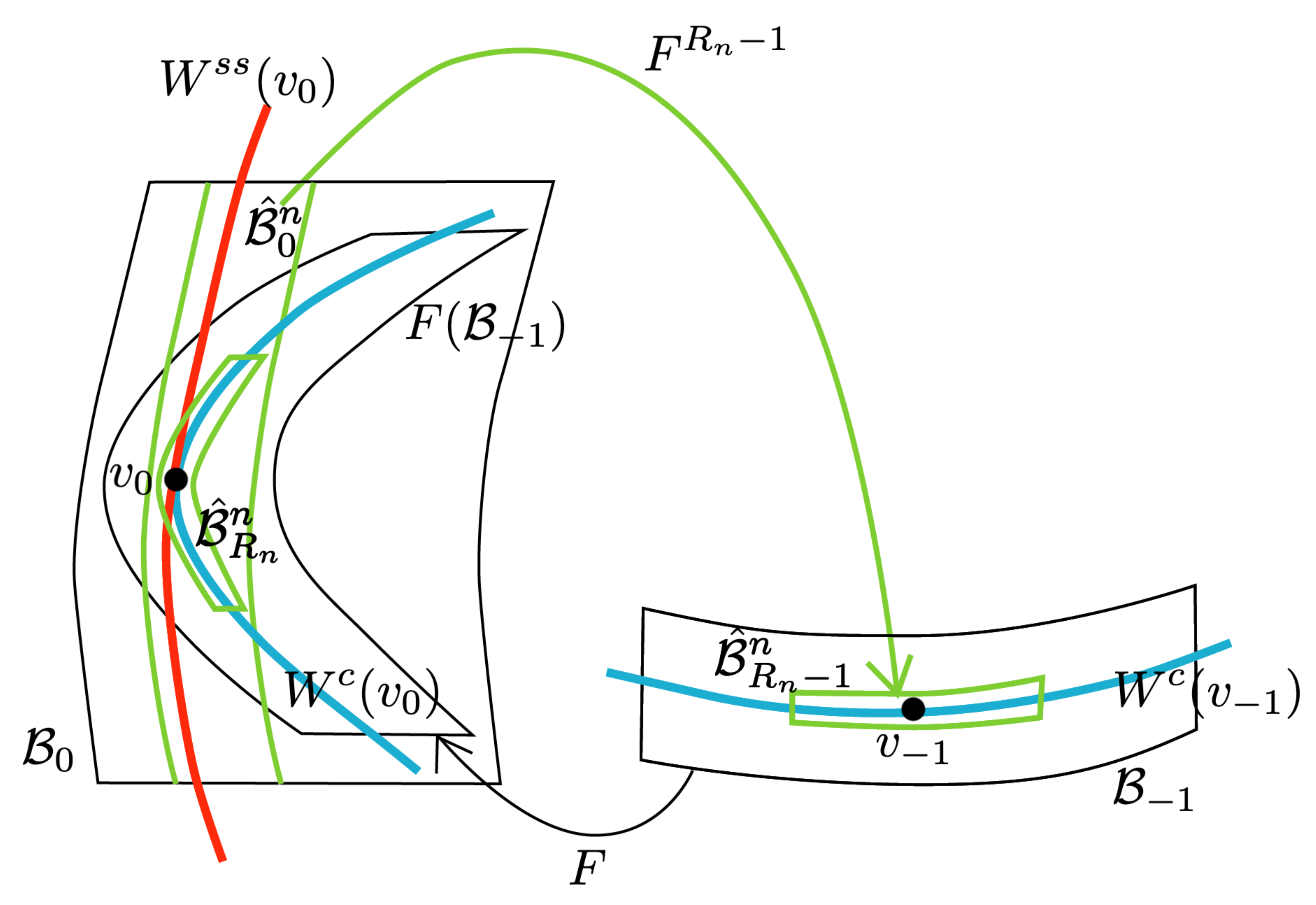}
\caption{Geometry near the critical value $v_0$ and the critical point $v_{-1}$ (if $N =\infty$). For $n \geq n_0$, we have $v_0 \in \hcB^n_0 \subset \cB_0$ and $v_{-1} \in \hcB^n_{R_n-1} \subset \cB_{-1}$. There exist charts $\Phi_0 : \cB_0 \to B_0$ and $\Phi_{-1} : \cB_{-1} \to B_{-1}$ such that $\Phi_0\circ F \circ \Phi_{-1}$ is H\'enon-like (see \eqref{eq.henon trans}). The charts $\Psi^n$ converges to $\Phi_0$}
\label{fig.chartconv}
\end{figure}

\begin{proof}
For $t \geq 0$ and $X \subset \bbR^2$, denote
$$
X(t) := \{p \in \bbR^2 \; | \; \dist(p, X) \leq t\}.
$$
Let
$$
\cB_0 := \cB^{n_0}_0(\lambda^{\bepsilon R_{n_0}})
\matsp{and}
\cC^n_0 := \hcB^n_0(\lambda^{\bepsilon R_n}) \setminus \hcB^n_0.
$$
By \eqref{eq.psi conv}, there exists a $C^r$-diffeomorphism $\Phi_0$ defined in a neighborhood of $\cZ_0$ such that
$$
\|\Phi_0\circ (\Psi^n)^{-1} - \Id\|_{C^r} < \lambda^{(1-\bepsilon)R_n}
\matsp{for all}
n_0 \leq n \leq N.
$$
Moreover, $\Phi_0$ can be extended a centered chart $\Phi_0 : (\cB_0, v_0) \to (B_0, 0)$ such that
$$
\Phi_0|_{\hcB^n_0 \setminus (\hcB^{n+1}_0 \cup \cC^{n+1}_0)} = \Psi^n|_{\hcB^n_0 \setminus (\hcB^{n+1}_0 \cup \cC^{n+1}_0)}
$$
and
$$
\|\Phi_0 \circ (\Psi^n|_{\cC^{n+1}_0})^{-1} -\Id \|_{C^r} <\lambda^{(1-\bepsilon)R_n}.
$$

Let $\cI^h_{-1} := \cI^N_{R_N-1} \ni v_{-1}$ if $N < \infty$, and $\cI^h_{-1} := W^c(v_{-1})$ if $N=\infty$. Observe that $F(\cI^h_{-1}) \ni v_0$ is a vertical quadratic curve in $\cB_0$. Denote
$$
J^v_0 := \pi_v\circ \Phi_0 \circ F(\cI^h_{-1}).
$$
Then there exists a $C^r$-map 
$$
f_0 : (J^v_0, 0) \to (\pi_h(B_0), 0)
$$
with a unique quadratic critical point at $0$ such that
$$
\Phi_0 \circ F(\cI^h_{-1}) = \{(f_0(y), y) \; | \; y \in J^v_0\}.
$$
The $C^r$-bound on $f_0$ follows from \propref{vert prop nest} and \lemref{hor curvs conv}.

Let $C \geq 1$ be the uniform constant given in \eqref{eq.proper depth 0}. Let
$$
D_0 := \{(f_0(y) + t, y) \in B_0 \; | \; |t| \leq \lambda {K_0}^{-1} \; \text{and} \; y \in J^v_0\},
$$
and
$$
\cB_{-1} := (\Phi_0\circ F)^{-1}(D_0).
$$
We define
$
\Phi_{-1} : (\cB_{-1}, v_{-1}) \to (B_{-1}, 0)
$
to be the unique chart satisfying
$$
\Phi_0 \circ F \circ \Phi_{-1}^{-1}(x,y) = (f_0(x) - \lambda y, x)
\matsp{for}
(x,y) \in B_{-1}.
$$

Consider the decomposition of $H_n$ given in \eqref{eq.decomp proof}. The second inequality in \eqref{eq.first entry} is given by \eqref{eq.en bound}. The upper bound in the first inequality follows immediately from \propref{cr lin comp}. For the lower bound, we observe that
$$
\|\left(\Phi_{p_{-R_n}}\circ (\Psi^n)^{-1}\right)^{-1}\|_{C^1} < \lambda^{-\bepsilon R_n}
$$
by \thmref{reg chart}. The lower bound now follows immediately from \thmref{reg chart} ii) and iii).
\end{proof}

Denote
\begin{equation}\label{eq.Ihv}
I^{h/v}_i := \pi_{h/v}(B_i)
\matsp{and}
\cI^h_i := \Phi_i^{-1}(I^h_i \times \{0\})
\matsp{for}
i \in \{0, -1\}.
\end{equation}
Observe that
$$
I^h_0 \Supset I^{n_0}_0 \Supset I^{n_0+1}_0 \Supset \ldots
\matsp{and}
I^h_{-1} \Supset h_{n_0}(I^{n_0}_0) \Supset h_{n_0+1}(I^{n_0+1}_0) \Supset \ldots.
$$
Moreover, if $X \subset \cB^n_0$, then \eqref{eq.first entry} implies
\begin{equation}\label{eq.first entry squeeze}
\Phi_{-1} \circ F^{R_n-1}(X) \subset h_n(I^n_0) \times [-\lambda^{(1-\bepsilon)R_n}, \lambda^{(1-\bepsilon)R_n}].
\end{equation}

We record the following consequences of \thmref{crit chart}.

\begin{lem}\label{est value}
Let $f_0 : I^h_{-1} \to I^h_0$ be the map with a unique critical point at $0$ given in \thmref{crit chart}. Then
$$
\frac{\kappa_F}{2} x^2 < f_0(x) < \frac{{K_0}}{2}x^2
\matsp{and}
\kappa_F |x| < |f_0'(x)| < {K_0} |x|.
$$
\end{lem}

\begin{lem}\label{root at value}
Let
$$
I^{h,\pm}_{-1} := \{x \in I^h_{-1} \; | \; \pm x > 0\}
\matsp{and}
g_\pm := \left(f_0|_{I^{h,\pm}_{-1}}\right)^{-1}.
$$
Denote $\theta := {K_0}/\kappa_F$. Then for $1 \leq i \leq r$, we have
$$
|g_\pm^{(i)}(t)| < \frac{\bar \theta}{|t|^{i-1/2}}
\matsp{for}
t > 0.
$$
\end{lem}

\begin{proof}
By \thmref{crit chart} and \lemref{est value}, we have
$$
\|f_0''\|_{C^{r-2}} < {K_0}
\comma
t < \frac{{K_0} x^2}{2}
\matsp{and}
|f_0'(x)| > \kappa_F |x|.
$$
The result now follows from \lemref{cr inverse}.
\end{proof}

\section{1D-Like Reduction}

For some $N \in \bbN \cup \{\infty\}$, let $F$ be the $N$-times regularly H\'enon-like renormalizable map considered in \secref{sec.conv chart}.

\subsection{Valuable projections}

Consider the valuable charts $\Phi_0$ and $\Phi_{-1}$ given in \thmref{crit chart}. Define $P_{-1} : (\cB_{-1}, v_{-1}) \to (I^h_{-1}, 0)$ and $P^n_0 : (\hcB^n_0, v_0) \to (I^n_0, 0)$ for $n_0 \leq n \leq N$ by
$$
P_{-1} := \pi_h \circ \Phi_{-1}
\matsp{and}
P^n_0 := \pi_h \circ \Psi^n.
$$
Denote
$$
I^n_{R_n-1} := P_{-1}(\hcB^n_{R_n-1}) = P_{-1}(\cI^n_{R_n-1})= h_n(I^n_0).
$$
Define the {\it $n$th (valuable) projection map} $\cP^n_0 : \hcB^n_0 \to \cI^n_0$ by
$$
\cP^n_0 := (\Psi^n)^{-1} \circ \Pi_h \circ \Psi^n.
$$
Observe that $\cP^n_0|_{\cI^n_0} = \Id$.

We record the following immediate consequence of \thmref{reg chart} and Propositions \ref{for gt} and \ref{back dt}.

\begin{lem}\label{flat}
For $n_0 \leq n\leq N$, denote $\lambda_n := \lambda^{(1-\bepsilon)R_n}$. Then for $0 < t < \lambda^{-\bepsilon R_n}$, the following statements hold.
\begin{enumerate}[i)]
\item Let $\tiE_{p_0} \in \bbP^2_{p_0}$ be a $t$-horizontal direction at $p_0 \in \hcB^n_0$. Then $\tiE_{p_{R_n-1}}$ is $(1+t)\lambda_n$-horizontal in $\cB_{-1}$.
\item Let $E_{p_{R_n-1}} \in \bbP^2_{p_{R_n-1}}$ be a $t$-vertical direction at $p_{R_n-1} \in \hcB^n_{R_n-1}$. Then $E_{p_0}$ is $t\lambda_n$-vertical in $\hcB^n_0$.
\item Let $\Gamma^h_0$ be a curve that is $t$-horizontal in $C^r$ in $\hcB^n_0$. Then $F^{R_n-1}(\Gamma^h_0)$ is $(1+t)^r\lambda_n$-horizontal in $C^r$ in $\hcB_{-1}$.
\item Let $\Gamma^v_{R_n-1}$ be a curve that is $t$-vertical in $C^r$ in $\hcB^n_{R_n-1}$. Then $F^{-R_n+1}(\Gamma^v_{R_n-1})$ is $t\lambda_n$-vertical in $C^r$ in $\hcB^n_0$.
\end{enumerate}
\end{lem}

By \lemref{flat} iii), $\cI^n_{R_n-1}$ is $\eta_n$-horizontal in $\cB_{-1}$. Thus, there exists a $C^r$-map $g_n : I^n_{R_n-1} \to \bbR$ with $\|g_n\|_{C^r} < \lambda_n$ such that
$$
\Phi_{-1}(\cI^n_{R_n-1}) = \{(x, g_n(x)) \; | \; x \in I^n_{R_n-1}\}.
$$
Define $G_n : I^n_{R_n-1} \to \Phi_{-1}(\cI^n_{R_n-1})$ by $G_n(x) := (x, g_n(x)).$ Define the {\it $n$th critical projection map $\cP^n_{-1} : P_{-1}^{-1}(I^n_{R_n-1}) \to \cI^n_{R_n-1}$} by
\begin{equation}\label{eq.-1 proj}
\cP^n_{-1} := \Phi_{-1}^{-1}\circ G_n \circ P_{-1}.
\end{equation}

\begin{lem}\label{first entry hor graph}
For $n_0 \leq n \leq N$, let $\Gamma_0$ be a horizontal curve in $\hcB^n_0$. Then 
$$
F^{R_n-1}|_{\Gamma_0} = (\cP^n_{-1}|_{\Gamma_{R_n-1}})^{-1}\circ F^{R_n-1} \circ \cP^n_0|_{\Gamma_0}.
$$
\end{lem}

\begin{proof}
Note that $\cP^n_{-1}$ is a projection along the vertical foliation $\cF^v_{-1}$ on $\cB_{-1}$, and $\cP^n_0$ is a projection along the vertical foliation on $\hcB^n_0$ obtained by pulling back $\cF^v_{-1}$ by $F^{-R_n+1}$. The claim follows immediately.
\end{proof}

\begin{figure}[h]
\centering
\includegraphics[scale=0.2]{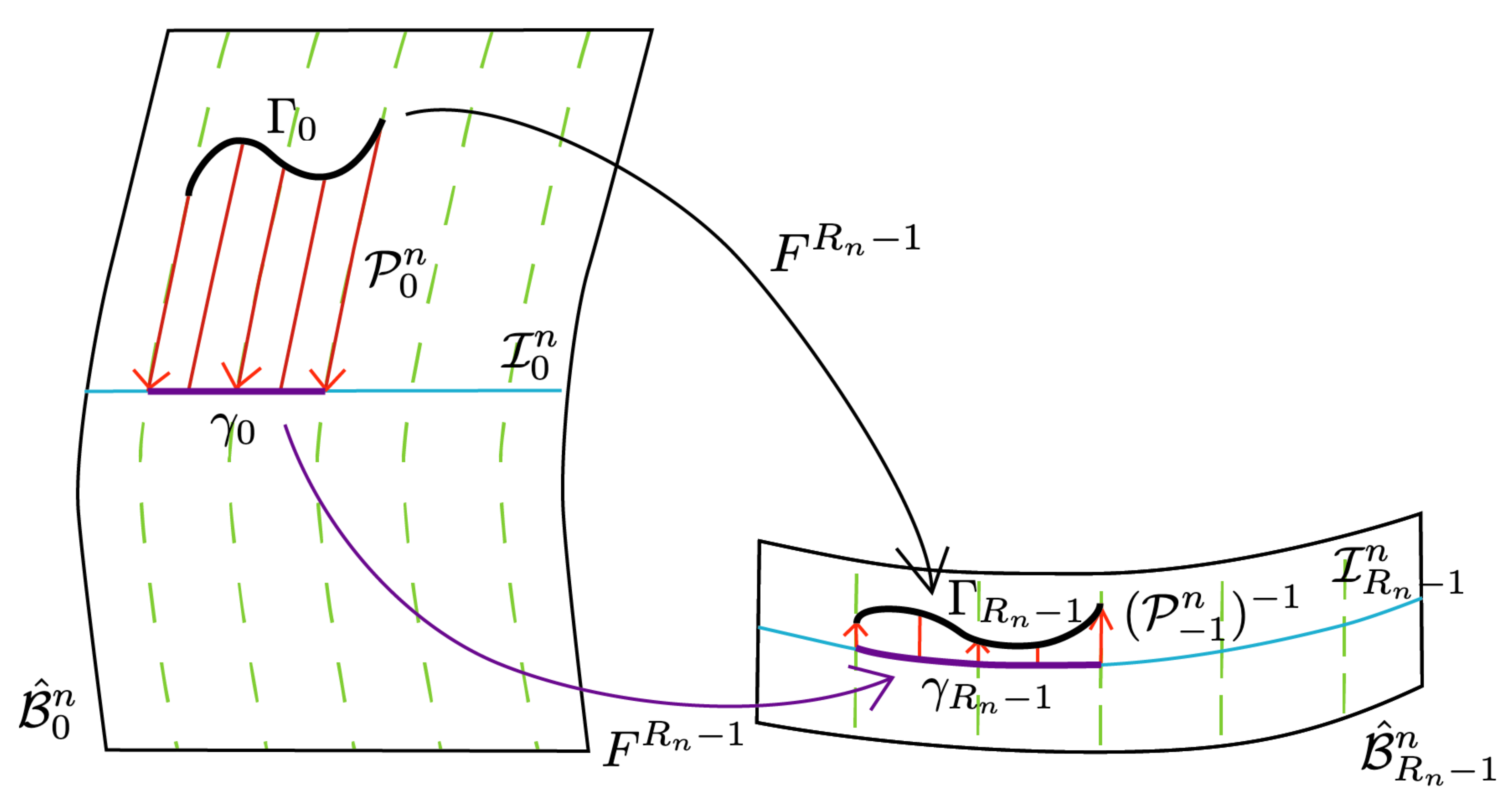}
\caption{Projections $\cP^n_0 : \hcB^n_0 \to \cI^n_0$ and $\cP^n_{-1} : \hcB^n_{R_n-1} \to \cI^n_{R_n-1}$ near the critical value $v_0$ and critical point $v_{-1}$ respectively. On any horizontal curve $\Gamma_0 \subset \hcB^n_0$, the iterate $F^{R_n-1}$ commutes with these projections.}
\label{fig.proj}
\end{figure}

\subsection{Passing near the critical value}

Let $K_0 \geq 1$ be the uniform constant given in \eqref{eq.const 0}. Assume that $n_0 \leq N$ is the smallest number such that
\begin{equation}\label{eq.proper depth 1}
\overline{K_1} \lambda^{\epsilon R_{n_0}} < \eta,
\end{equation}
where
\begin{equation}\label{eq.const 1}
K_1 = K_1({K_0}, \kappa_F) \geq 1
\end{equation}
is a uniform constant.

Let  $z = (a,b) \in B_0 = I^h_0 \times I^v_0$. For $t \geq 0$, define
$$
V_z(t) := [a - t, a+t] \times I^v_0.
$$
For $p \in \hcB^n_0$ and $t \geq 0$, denote
$$
\cV_p^n(t) := (\Psi^n)^{-1}(V_{\Psi^n(p)}(t)) \subset \hcB^n_0.
$$

We record the following immediate consequences of Lemmas \ref{est value} and \ref{root at value}, and  \eqref{eq.proper depth 1}.

\begin{lem}\label{quad flat}
For $n_0 \leq n \leq N$, let $E_{p_{-1}} \in \bbP^2_{p_{-1}}$ be a $\lambda^{\bepsilon R_n}$-horizontal direction at $p_{-1} \in \cB_{-1}$. If
$$
p_0 \in \hcB^n_0 \setminus \cV^n_{v_0}(\lambda^{\bepsilon R_n})
$$
then $E_{p_0}$ is $\lambda^{-\bepsilon R_n}$-horizontal in $\hcB^n_0$. Similarly, let $\Gamma_{-1}$ be $\lambda^{\bepsilon R_n}$-horizontal curve in $\cB_{-1}$. If
$$
\Gamma_0:= F(\Gamma_{-1}) \subset \hcB^n_0 \setminus \cV^n_{v_0}(\lambda^{\bepsilon R_n})
\matsp{with}
t > \lambda^{\bepsilon R_n},
$$
then $\Gamma_0$ is $\lambda^{-\bepsilon R_n}$-horizontal in $C^r$ in $\hcB^n_0$.
\end{lem}

\begin{lem}\label{quad straight}
For $n_0 \leq n \leq N$, let $\tiE_{p_0} \in \bbP^2_{p_0}$ be a $\lambda^{\bepsilon R_n}$-vertical direction at $p_0 \in \hcB^n_0$. If
$$
p_0 \in \hcB^n_{R_n} \setminus \cV^n_{v_0}(\lambda^{\bepsilon R_n}),
$$
then $\tiE_{p_0}$ is $\lambda^{-\bepsilon R_n}$-vertical in $\cB_{-1}$. Similarly, let $\tiGamma_0$ be $\lambda^{\bepsilon R_n}$-vertical curve in $\hcB^n_0$. If
$$
\tiGamma_0 \subset \hcB^n_{R_n} \setminus \cV^n_{v_0}(\lambda^{\bepsilon R_n}),
$$
then $\tiGamma_{-1} := F^{-1}(\tiGamma_0)$ is $\lambda^{-\bepsilon R_n}$-vertical in $C^r$ in $\cB_{-1}$.
\end{lem}


\subsection{1D-Like Structure}\label{subsec.combin}

Henceforth, suppose that the combinatorics of renormalizations of $F$ are of $\bfb$-bounded type for some integer $\bfb \geq 2$. Moreover, assume that $\epsilon$ is sufficiently small so that $\bfb\bepsilon < 1$. By only considering every other returns if necessary, we may also assume without loss of generality that $r_n := R_{n+1}/R_n \geq 3$ for ${n_0}\leq n \leq N$.

Let $z = (a,b), w = (c,d) \in B_0 = I^h_0 \times I^v_0$. Denote
$$
m := \min\{a, c\}
\matsp{and}
M := \max\{a, c\}.
$$
For $t \geq 0$, define
$$
V_{[z,w]}(t) := [m -t, M + t] \times I^v_0.
$$
For $n_0 \leq n \leq N$; $p, q \in \hcB^n_0$ and $t \geq 0$, denote
$$
\cV_{[p, q]}^n(t) := (\Psi^n)^{-1}(V_{[\Psi^n(p), \Psi^n(q)]}(t)).
$$

Let $s \in \{0, 1, 2\}$. For $n_0 \leq n \leq N-s$ and $k \geq -1$, denote
$$
a^n_k := P^n_0(v_{kR_n})
\matsp{and}
b^{n,s}_k:= P^n_0(v_{kR_n+R_{n+s}}) = a^n_{k+R_{n+s}/R_n}.
$$
Define
$$
\hB^{n,s}_{kR_n} := V_{[a^n_k, b^{n,s}_k]}(\lambda^{\bepsilon R_n})
\matsp{and}
\hcB^{n,s}_{kR_n} := (\Psi^n)^{-1}(\hB^{n,s}_{kR_n}).
$$
In particular, we have
$$
\hcB^n_0 \supset \hcB^{n,0}_0 := \cV^n_{[v_0, v_{R_n}]}(\lambda^{\bepsilon R_n}).
$$
See \figref{fig.valuestruct}.

\begin{defn}
For $n_0 \leq n < N$ and $1 \leq s \leq N-n$, we say that the H\'enon-like return $(F^{R_n}, \Psi^n)$ has {\it 1D-like structure of depth $s$} if:
\begin{enumerate}[i)]
\item $\hcB^{n, s}_{lR_n} \cap \hcB^{n, s}_{kR_n} = \varnothing$ for $0 \leq l,k < R_{n+s}/R_n$ with $l \neq k$;
\item $a^n_0 = 0< b^{n, s}_0< a^n_k  \, ,\, b^{n, s}_k< b^{n,s}_1  < a^n_1  $ for $2 \leq k < R_{n+s}/R_n$; and
\item $F^{R_n}(\hcB^{n,s}_{kR_n}) \Subset \hcB^{n,s}_{(k+1)R_n \md{R_{n+s}}}$ for $0 \leq k < R_{n+s}/R_n$.
\end{enumerate}
\end{defn}

\begin{prop}\cite[Proposition 6.5]{Y}\label{value struct bound}
Let $n_0 \leq n \leq N$. Suppose that $F^{R_n}|_{\cB^n_0}$ is twice non-trivially topological renormalizable with combinatorics of $\bfb$-bounded type. Then for $m = n-s$ with $s = O(1)$, the H\'enon-like return $(F^{R_m}, \Psi^m)$ has 1D-like structure of depth $s$. In particular, $\hcB^{n, 0}_0$ is $R_n$-periodic.
\end{prop}

\begin{figure}[h]
\centering
\includegraphics[scale=0.25]{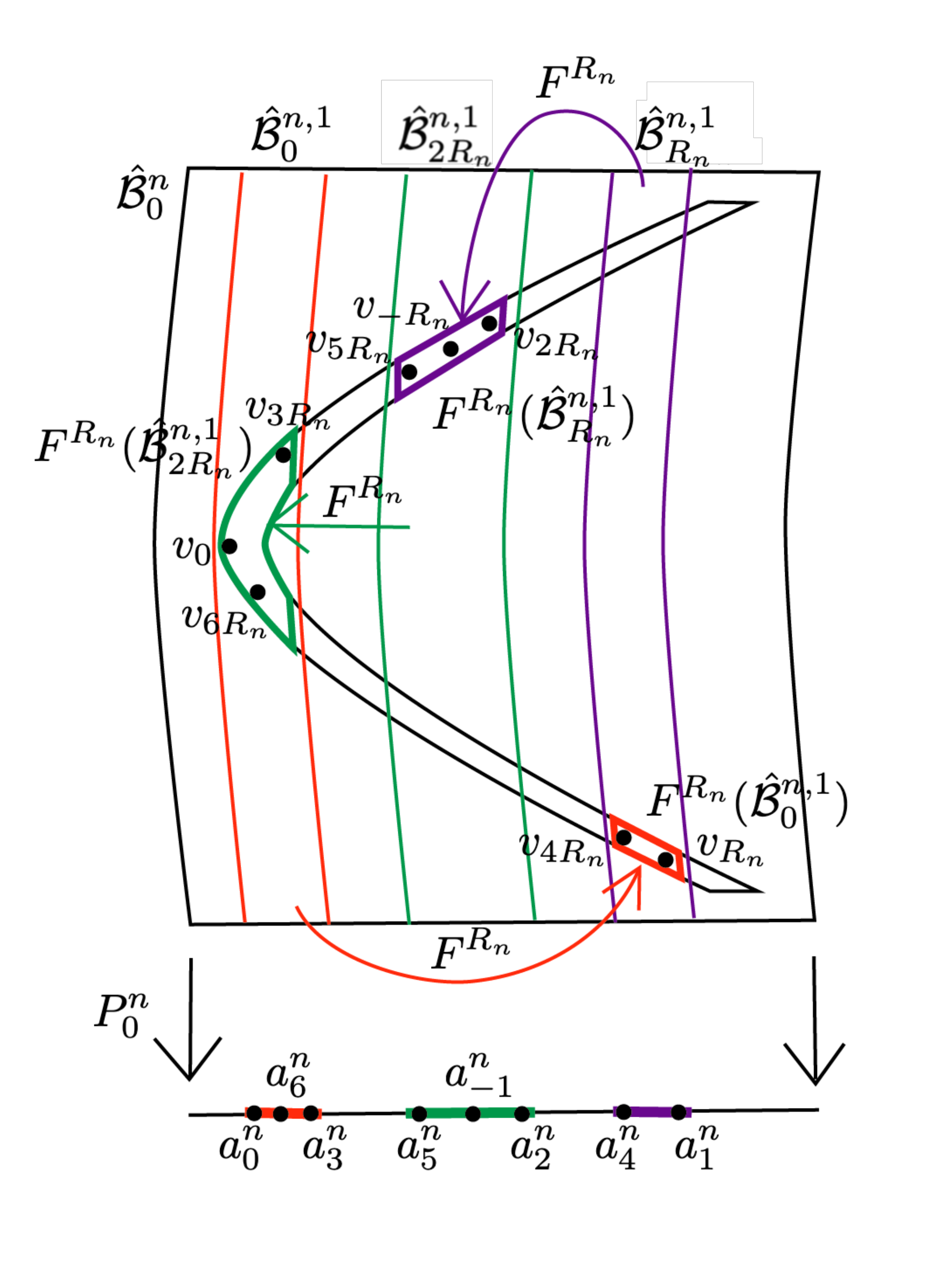}
\caption{The 1D-like structure of depth $1$ of $(F^{R_n}, \Psi^n)$ for $n \geq n_0$ (for $r_n := R_{n+1}/R_n = 3$). The $R_{n+1}$-periodic domains $\hcB^{n,1}_0$, $\hcB^{n,1}_{R_n}$ and $\hcB^{n,1}_{2R_n}$ containing $v_0$, $v_{R_n}$ and $v_{2R_n}$ respectively are vertically proper and pairwise disjoint in $\hcB^n_0$. Moreover, $F^{R_n}(\hcB^{n,1}_{kR_n}) \Subset \hcB^{n,1}_{(k+1)R_n \md{R_{n+1}}}$. Under the projection $P^n_0 : \hcB^n_0 \to I^n_0 \subset \bbR$, the orbit $\{v_{kR_n}\}_{k=-1}^{r_n}$ of the critical value are mapped to $\{a^n_k\}_{k=-1}^{r_n}$.}
\label{fig.valuestruct}
\end{figure}

By \propref{value struct bound}, we may henceforth assume without loss of generality that for all $n_0 \leq n \leq N$ such that $F^{R_n}|_{\cB^n_0}$ is twice non-trivially renormalizable, we have
\begin{equation}\label{eq.min domain}
\cB^n_0 := \hcB^{n,0}_0 := \cV^n_{[v_0, v_{R_n}]}(\lambda^{\bepsilon R_n}).
\end{equation}

\begin{lem}\label{unproject}
Let $n_0 \leq n \leq N$. Suppose that $F^{R_n}|_{\cB^n_0}$ is twice topological renormalizable with combinatorics of $\bfb$-bounded type. For $s\in \{1, 2\}$ and $m := n-s$, let $\Gamma_0$ be a $\lambda^{-\bepsilon R_m}$-horizontal curve in $\cB^n_0$. Then the following statements hold for $1 \leq k \leq R_n/R_m$:
\begin{enumerate}[i)]
\item $\Gamma_{(k-1)R_m}$ is $\lambda^{-\bepsilon R_m}$-horizontal in $\cB^m_0$; and
\item $\Gamma_{kR_m-1}$ is $\lambda^{(1-\bepsilon)R_m}$-horizontal in $\cB_{-1}$.
\end{enumerate}
\end{lem}

\begin{proof}
The result is an immediate consequence of Lemmas \ref{flat} iii) and \ref{quad flat}, and \propref{value struct bound}.
\end{proof}


\section{Critical Recurrence}\label{sec.crit rec}

Let $F$ be the infinitely regularly H\'enon-like renormalizable map with combinatorics of bounded type considered in \subsecref{subsec.combin} (with $N=\infty$). In this section, we prove the following result.

\begin{thm}\label{crit rec}
We have
$$
\cZ_0 := \bigcap_{n=1}^\infty \cB^n_{R_n} = \{v_0\}.
$$
Consequently, the orbit of $v_0$ is recurrent.
\end{thm}

\begin{proof}
Let
$$
\cY_0 := \bigcap_{n=1}^\infty \cB^n_0
\comma
\cI^\infty_0 := \cI^h_0 \cap \cY_0
\matsp{and}
I^\infty_0 := \pi_h\circ \Phi_0(\cI^\infty_0).
$$
Note that every point $p_0 \in \cY_0$ is infinitely forward $(L, \epsilon, \lambda)$-regular. Moreover, arguing as in the proof of \propref{vert prop nest}, we see that $W^{ss}(p_0)$ is vertically proper in $\cB^1_0$. Let $W^{ss}_{\loc}(p_0)$ be the connected component of $W^{ss}(p_0) \cap \cB^1_0$ containing $p_0$. Then we have
$$
\cY_0 = \bigcup_{p_0 \in \cI^\infty_0} W^{ss}_{\loc}(p_0).
$$
Since $\cY_i := F^i(\cY_0) \subset \cB^n_i$ for all $n \in \bbN$ such that $0\leq i < R_n$, we see that
$$
\cY_i \cap \cY_0 = \varnothing
\matsp{for}
i \in\bbN.
$$
We claim that $\cY_0 = W^{ss}_{\loc}(v_0)$. Suppose towards a contradiction that this is not true. By \eqref{eq.henon trans} and \eqref{eq.first entry squeeze}, this means that there exists a uniform constant $b > 0$ such that $(0, b) \subset I^\infty_0$.

Recall that for $n \in \bbN$, the curve $\cI^n_{R_n}$ is vertical quadratic in $\cB^n_0$. Let $c^n_0 \in \cI^n_0$ be the unique point such that
$$
E^{gv}_{\Psi^n\left(c^n_{R_n}\right)} = D(\Psi^n\circ F^{R_n})\left(E^h_{c^n_0}\right).
$$
By \thmref{crit chart} ii), we see that $\Phi_{-1}(\cI^n_{R_n-1})$ is $\lambda^{(1-\bepsilon)R_n}$-horizontal in $\cB_{-1}$. Hence, by \thmref{crit chart} i), we have
$$
\|c^n_{R_n} -v_0\| < \lambda^{(1-\bepsilon)R_n}.
$$

Let $M \in \bbN$ be sufficiently large so that for $n \geq M$, we have
$$
 P^n_0(c^n_{R_n})<\lambda^{(1-\bepsilon)R_n}<\lambda^{\bepsilon R_M}<b/2.
$$
Note that for $0\leq k < R_n/R_M$, we have
$$
\cB^n_{kR_M} \cap \cB^n_0 = \varnothing.
$$
Thus, applying \lemref{quad flat} and proceeding by induction, we see that the curve $\cI^n_{k R_M}$ is $\lambda^{-\bepsilon R_M}$-horizontal in $\cB^M_0$, and $\cI^n_{(k+1)R_M-1}$ is $\lambda^{(1-\bepsilon)R_M}$-horizontal in $\cB_{-1}$.

Define $\cB^n_{-kR_M}$ for $0\leq k < R_n/R_M$ inductively as follows. Suppose that
\begin{itemize}
\item $F^{kR_M}(\cB^n_{kR_M}) \Subset \cB^n_0$;
\item $\Psi^M(\cB^n_{-kR_M})$ is a vertically proper quadrilateral in $B^M_0$, whose side boundaries are $\lambda^{(1-\bepsilon)R_M}$-vertical; and
\item $\cB^n_{-kR_M} \supset \cI^n_{R_n-kR_M}$.
\end{itemize}
Since $\cB^n_0$ is $R_n$-periodic (see \propref{value struct bound} iii)), property i) implies that
$$
\cB^n_{-kR_M} \cap \cB^n_{-iR_M} = \varnothing
\matsp{for}
0\leq i < k.
$$
This, together with property ii) ensure that
$$
F^{-1}(\cB^n_{-kR_M}) \cap \cB^M_{R_M-1}
$$
consists of exactly two connected components (unless $k=0$, in which case there is only one connected component). Let $\cB^n_{-kR_M-1}$ be the component containing $\cI^n_{R_n-kR_M-1}$. Define
$$
\cB^n_{-(k+1)R_M} := F^{-R_M+1}(\cB^n_{-kR_M-1}).
$$
By \lemref{quad straight}, we see that
$$
\partial \cB^n_{-kR_M-1} \setminus \partial \cB^M_{R_M-1}
$$
consists of two $\lambda^{-\bepsilon R_M}$-vertical curves $\Gamma^{n, \pm}_{-kR_M-1}$ in $\cB_{-1}$, and
$$
\Gamma^{n,\pm}_{-(k+1)R_M} := F^{-R_M+1}(\Gamma^{n, \pm}_{-kR_M-1})
$$
are $\lambda^{(1-\bepsilon)R_M}$-vertical and vertically proper in $\cB^M_0$.

Since the sets
$$
\cB^n_{-(k+1)R_M} \supset \cI^n_{R_n-(k+1)R_M}
\matsp{for}
0\leq k < R_n/R_M
$$
are disjoint, the intervals
$$
I^n_{kR_M} := P^M_0(\cI^n_{kR_M})
$$
must be disjoint in $I^M_0$.

Consider the diffeomorphism $h_M$ given in \thmref{crit chart} ii). Define
$$
g_k^n(x):= \cP^M_0 \circ F \circ (\cP^M_{-1}|_{\cI^n_{(k+1)R_M-1}})^{-1}(h_M(x),0)
\matsp{for}
x \in I^n_{kR_M}.
$$
Since $\cI^n_{(k+1)R_M-1}$ and $\cI^n_{(k+1)R_M}$ are uniformly horizontal in $\cB_{-1}$ and $\cB_0$ respectively, it follows that $\|g_k^n\|_{C^r}$ is uniformly bounded. Moreover, 
$$
\sum_{k=0}^{R_n/R_M-1} |I^n_{kR_M}| < |I^M_0|.
$$
Thus, we conclude from \thmref{denjoy} that
$$
G^n:= g_{R_n/R_M-1}^n\circ \ldots \circ g_0^n
$$
has uniformly bounded distortion.

Let
$$
I^{n+1}_{-R_n} = P^M_0(\cB^{n+1}_{-R_n}).
$$
Then $I^{n+1}_{-R_n}$ and $I^{n+1}_0$ are disjoint intervals in $I^n_0$. Moreover, $|I^{n+1}_0|$ is uniformly bounded below, while
$$
|I^{n+1}_{-R_n}|, |I^{n+1}_{R_n}| \to 0
\matsp{as}
n \to \infty.
$$
However,
$$
G^n(I^{n+1}_{-R_n}) = I^{n+1}_0
\matsp{and}
G^n(I^{n+1}_0) = I^{n+1}_{R_n}.
$$
This contradicts the fact that $G^n$ has uniformly bounded distortion. The result follows.
\end{proof}


 \section{A Priori Bounds}\label{sec.a priori}

Let $r \geq 2$ be an integer, and consider a $C^{r+4}$-H\'enon-like map $F : B \to B$. For some $N \in \bbN \cup \{\infty\}$; $L \geq 1$ and $\epsilon, \lambda \in (0,1)$, suppose that $F$ has $N$ nested $(L, \epsilon, \lambda)$-regular H\'enon-like returns given by \eqref{eq.returns} with combinatorics of $\bfb$-bounded type for some integer $\bfb \geq 3$. By only considering every other returns if necessary, we may also assume without loss of generality that $r_n := R_{n+1}/R_n \geq 3$ for ${n_0}\leq n \leq N$. Assume that $\epsilon$ is sufficiently small so that $\bfb\bepsilon < 1$. Also assume that $N$ is sufficiently large, so that for some smallest number $0 \leq n_0 \leq N$, we have \eqref{eq.proper depth 1}. Lastly, suppose that $F^{R_N}|_{\cB^N_0}$ is twice non-trivially topologically renormalizable (so that the H\'enon-like returns of $F$ have 1D-like structure of depth $2$ by \propref{value struct bound}).

\subsection{The outline of strategy}

For $n_0 \leq n \leq N$, consider the horizontal cross-section of the $n$th renormalization domain $\cB^n_0$:
$$
\cI^n_0 := (\Psi^n)^{-1}(I^n_0 \times \{0\}) =\cI^h_0 \cap \cB^n_0 \ni v_0.
$$
See \eqref{eq.Ihv}. We want to prove that $\Dis(F^{R_n}, \cI^n_0)$ is uniformly bounded.

The general strategy is to reduce the 2D dynamics of $F$ acting on $\cI^n_0$ to a 1D mapping scheme for which standard 1D arguments can be applied to control distortion. Below we give a broad description of this 1D mapping scheme using simpler notations to better convey the main ideas. In the actual proof, the 1D scheme is derived from the 2D dynamics it is modeling, which forces the notations to become more complicated.

Fix some intervals $I_0, I_{-1} \subset \bbR$. For $1 \leq n \leq N$, let $\{A^n_i\}_{i=0}^{R_n-1}$ and $\{\chA^n_i\}_{i=0}^{R_n-1}$ be collections of pairwise disjoint subintervals in $I_0$ and $I_{-1}$ respectively so that $0 \in \chA^n_{R_n-1} \subset I_{-1}$. Consider the following mapping scheme for $1 \leq i \leq R_n$:
\begin{itemize}
\item a $C^2$-diffeomorphism
$$
\phi^n_i : A^n_{i-1} \to \chA^n_{i-1} := \phi^n_i (A^n_{i-1})
$$
with uniformly bounded $C^2$-norm; and
\item a quadratic power map 
$$
g^n_i : \chA^n_{i-1} \to A^n_i := g^n_i(\chA^n_{i-1})
$$
given by
$$
g^n_i(x) = x^2 + a^n_i
\matsp{for some}
a^n_i \in \bbR.
$$
\end{itemize}
Define
$$
H_i := \phi^n_i \circ g^n_{i-1} \circ \phi^n_{i-1} \circ \ldots \circ g^{n}_1 \circ \phi^{n}_1.
$$

Suppose that the domains of $\phi^n_i$ for $1\leq i\leq R_n$ can be extended so that $g^n_i \circ H_i$ maps a strictly larger interval $\tiA^n_{0, i} \Supset A^n_0$ diffeomorphically onto an image $g^n_i \circ H_i(\tiA^n_{0, i})$ that contains the two adjacent neighbors $A^n_{\iota_-(i)}$ and $A^n_{\iota_+(i)}$ of $A^n_i$ (or at least subintervals in $A^n_{\iota_\pm(i)}$ of commensurate lengths). Then we can apply Koebe distortion theorem (see \secref{sec.koebe}) to conclude that $H_{R_n}|_{A^n_0}$ has uniformly bounded distortion. 

We now give a brief description of how the 2D dynamics of $F$ acting on the curve $\cI^n_0$ is reduced to the above 1D mapping scheme. The main idea is to weave into the dynamics of $F$ systematic applications of projections near the critical value $v_0$. This confines the orbit of $\cI^n_0$ to lie in a fixed union of curves $\{\cI^{n_0}_i\}_{i=0}^{R_{n_0}-1}$. These projections are then ``undone’’ near the critical point $v_{-1}$ to recover the original dynamics. See the definitions of the maps $H^n_i$ and $\hH_i$, as well as Lemmas \ref{H diffeo near tip} and \ref{lem unproj}. See also \figref{fig.hdiffeo}.

The pairwise disjointedness of the collection of images $\{\cJ^n_i\}_{i=0}^{R_n-1}$ of $\cI^n_0$ under $\hH_i$ relies on the 1D combinatorial structures of the renormalizations of the 2D map $F$ established in \secref{subsec.combin}. See \figref{fig.valuestruct} and \lemref{J order}.

Contributions by quadratic power maps in the composition $\hH_i$ arise in the following way. When the inverse projection is applied near the critical point $v_{-1}$, it is onto a nearly horizontal curve (approximating a subarc of a center manifold of $v_{-1}$). Under one iterate of $F$, this curve is mapped to a subarc of a vertical quadratic curve near the critical value $v_0$. Then projecting along a vertical foliation to the transverse horizontal arc $\cI^{n_0}_0 \supset \cI^n_0$ produces the effect of applying a quadratic power map. See \propref{H decomp fine}.

Lastly, the extension of $\hH_i$ to a strictly larger domain $\ticI^n_{0, i} \Supset \cI^n_0$ so that the image $\hH_i(\ticI^n_{0, i})$ covers (commensurate portions of) the adjacent neighbors $\cJ^n_{\iota_\pm(i)}$ of $\cJ^n_i$ is done in Propositions \ref{cK} and \ref{ticI}.

\subsection{The proof of the Main Theorem}

First, we need the following lemma (which requires the 3 additional degrees of smoothness assumed in this section). Recall that $P^n_0 := \pi_h \circ \Psi^n$ for $n_0 \leq n \leq N$.

\begin{lem}\label{quad mapping}
Let $\kappa_F, K_1 > 0$ be the constants given in \thmref{crit chart} and \eqref{eq.const 1} respectively. Consider a $C^{r+3}$-map $g : I \to \bbR$ on an interval $I \subset I^h_{-1}$ such that $\|g\|_{C^2} < \underline{\kappa_F}$. Denote $G(x) := (x,g(x)).$ Then there exist $a \in I^h_0$ and a $C^r$-diffeomorphism $\psi: I \to \psi(I)$ with
$$
\|\psi^{\pm 1}\|_{C^r} < K_1(1+\|g\|_{C^{r+3}})
$$
such that we have
\begin{equation}\label{eq.quad map}
Q(x) := P^n_0 \circ F \circ \Phi_{-1}^{-1}\circ G(x) = \kappa_F \cdot (\psi(x))^2+a
\end{equation}
where defined.
\end{lem}

\begin{proof}
By \thmref{crit chart} i), it suffices to show that there exists $\tipsi_g$ with
$$
\|\tipsi^{\pm 1}\|_{C^r} < {K_1}(1+\|g\|_{C^{r+3}})
$$
such that
$$
\tiQ(x) := \pi_h \circ \Phi_0 \circ F \circ \Phi_{-1}^{-1} \circ G(x) = \kappa_F \cdot (\tipsi(x))^2 +\tia.
$$
By \eqref{eq.henon trans}, we have $\tiQ = f_0 - \lambda \cdot g.$ By the bound on $\|g\|_{C^2}$, we see that $\tiQ$ has a unique critical point, $\tiQ''(x)$ is bounded below by $c\kappa_F$ for some uniform constant $c > 0$, and
$$
\|\tiQ''\|_{C^{r+1}} < {K_1}(1+\|g\|_{C^{r+3}}).
$$
The result now follows from \lemref{factor}.
\end{proof}

For $n_0 \leq n \leq N$, define a sequence of maps $\{H^n_i\}_{i=0}^\infty$ as follows. First, let $H^{n_0}_i := F^i.$ Proceeding inductively, suppose $H^{n-1}_i$ is defined. Write $i = j+kR_n$ with $k \geq 0$ and $0\leq j < R_n$. Define
$$
H^n_i := H^{n-1}_j \circ \cP^n_0 \circ F^{kR_n},
$$
where
$$
\cP^n_0 := (\Psi^n)^{-1} \circ \Pi_h \circ \Psi^n
$$
is the $n$th projection map near the critical value $v_0$. Observe that $H^n_i$ is well-defined on $\cB^n_0$.

\begin{lem}\label{H diffeo near tip}
Let $s \in \{1, 2\}$ and $n_0 \leq n \leq N-s$. Then $H^n_i|_{\cI^{n+s}_1}$ is a diffemorphism for $0\leq i < R_{n+s}$.
\end{lem}

\begin{proof}
The statement is clearly true for $n = {n_0}$. Suppose the statement is true for $n -1$. If $i < R_n$, then
$$
H^n_i|_{\cI^{n+s}_1} = H^{n-1}_i|_{\cI^{n+s}_1}
$$
is a diffeomorphism. Suppose the same is true for $i < (k-1)R_n$ with $2 \leq k < R_{n+s}/R_n$. Observe that
$$
H^n_{kR_n} =\cP^n_0 \circ F^{kR_n}.
$$
By \lemref{unproject} i), the map $\cP^n_0|_{\cI^{n+s}_{kR_n}}$ is a diffeomorphism. For $i = j + kR_n$ with $j < R_n$, we have
$$
H^n_i := H^{n-1}_j \circ \cP^n_0 \circ F^{kR_n}.
$$
Since
$$
\cP^n_0(\cI^{n+s}_{kR_n}) \subset \cI^n_0,
$$
the result follows.
\end{proof}

Recall the definition of $\cP^n_{-1}$ for $n_0 \leq n \leq N$ given in \eqref{eq.-1 proj}. 

\begin{lem}\label{lem unproj}
For $s\in \{1, 2\}$ and $n_0 \leq n \leq N-s$, let $\Gamma_0$ be a $C^r$-curve which is $\lambda^{-\bepsilon R_n}$-horizontal in $\cB^{n+s}_0$. Then for $1 \leq k \leq R_{n+s}/R_n$, we have
$$
F^{kR_n-1}|_{\Gamma_0} = \left(\cP^{n_0}_{-1}|_{\Gamma_{kR_n-1}}\right)^{-1}\circ H^n_{kR_n-1}|_{\Gamma_0}.
$$
\end{lem}

\begin{proof}
If $k=1$, then the result follows immediately from \lemref{first entry hor graph}. Suppose the result is true for some $n_0 \leq n < N-s$ and $1 \leq k < R_{n+s}/R_n$. By definition, we have
$$
H^n_{(k+1)R_n-1} = H^n_{kR_n-1} \circ F^{R_n}.
$$
If $\Gamma_0$ is a $C^r$-curve which is $\lambda^{-\bepsilon R_n}$-horizontal in $\cB^{n+s}_0$, then by \lemref{unproject} i), we see that
$
\Gamma_{R_n} := F^{R_n}(\Gamma_0)
$
is a $C^r$-curve which is $\lambda^{-\bepsilon R_n}$-horizontal in $\cB^n_0$. Thus, by induction, we have
$$
F^{kR_n-1}|_{\Gamma_{R_n}} = \left(\cP^{n_0}_{-1}|_{\Gamma_{(k+1)R_n-1}}\right)^{-1}\circ H^n_{kR_n-1}|_{\Gamma_{R_n}}.
$$
Composing on the right by $F^{R_n}|_{\Gamma_0}$, the result is true in this case.

Finally, suppose that the result is true for some $n_0 \leq n < N-s$ and $k = R_{n+1}/R_n$. Let $\gamma_0 := \cP^{n+1}_0(\Gamma_0)$. By the induction hypothesis, we have:
$$
F^{R_{n+1}-1}|_{\gamma_0} = \left(\cP^{n_0}_{-1}|_{\gamma_{R_{n+1}-1}} \right)^{-1} \circ H^n_{R_{n+1}-1}|_{\gamma_0}.
$$
Applying \lemref{first entry hor graph}:
\begin{align*}
F^{R_{n+1}-1}|_{\Gamma_0} &= \left(\cP^{n+1}_{-1}|_{\Gamma_{R_{n+1}-1}} \right)^{-1} \circ \left(\cP^{n_0}_{-1}|_{\gamma_{R_{n+1}-1}} \right)^{-1} \circ H^n_{R_{n+1}-1} \circ \cP^{n+1}_0|_{\Gamma_0}\\
&= \left(\cP^{n_0}_{-1}|_{\Gamma_{R_{n+1}-1}} \right)^{-1} \circ H^{n+1}_{R_{n+1}-1}|_{\Gamma_0}.
\end{align*}
\end{proof}

\begin{figure}[h]
\centering
\includegraphics[scale=0.2]{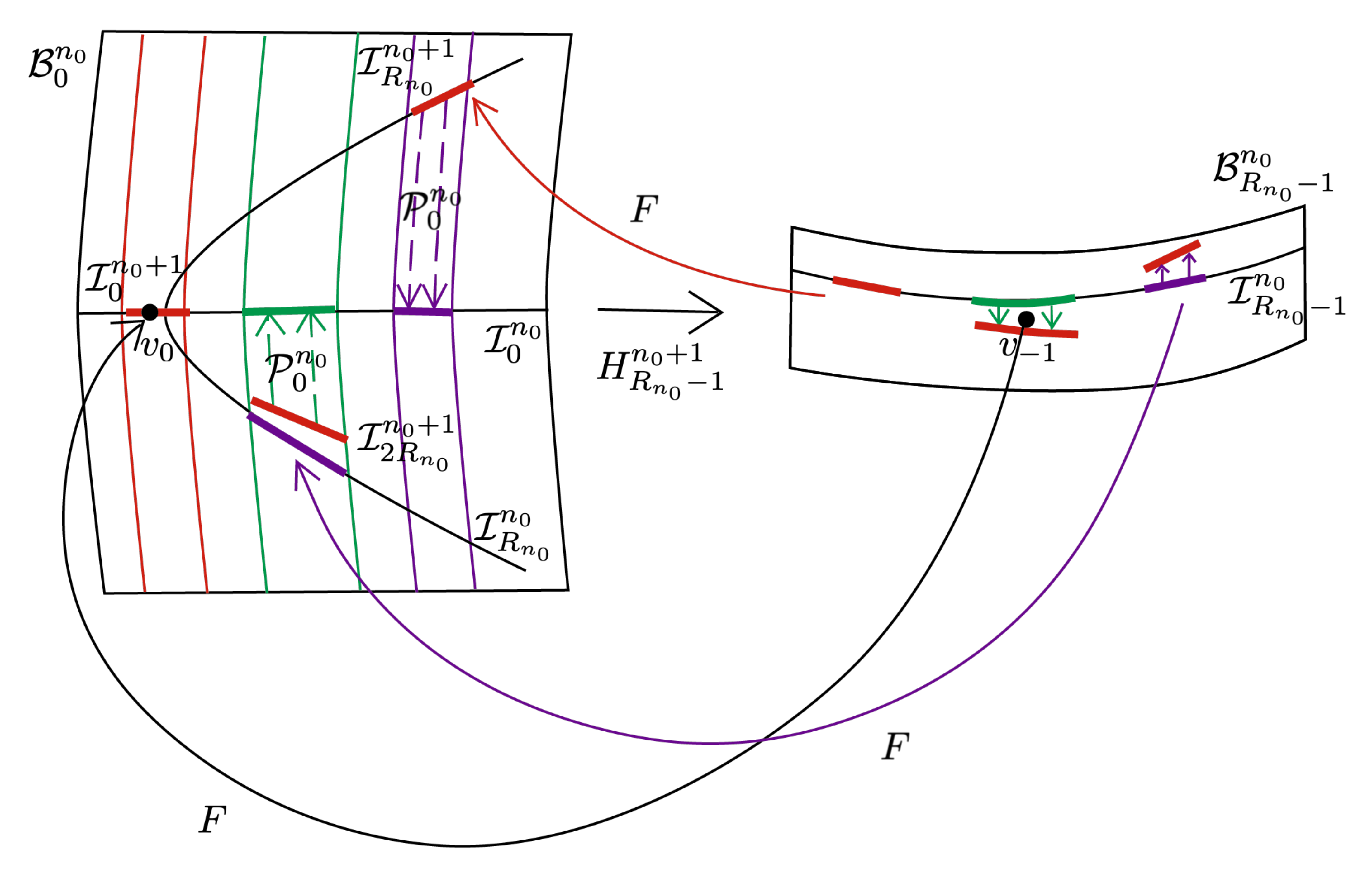}
\caption{Visualization of the map $H^{n_0}_i$ for $0 \leq i < R_{n_0+1}$ acting on the horizontal curve $\cI^{n_0+1}_0 \subset \cI^{n_0}_0$ (for $r_{n_0} := R_{n_0+1}/R_{n_0} = 3$). The orbit of $\cI^{n_0+1}_0$ makes returns to $\cB^{n_0}_0 \ni v_0$ under $F^{kR_n}$ for $0 \leq k < r_{n_0}$. At these moments, the projection map $\cP^{n_0}_0$ is applied to $\cI^{n_0+1}_{kR_{n_0}}$ to bring it down to $\cI^{n_0}_0$. These projections are then ``undone'' in $\cB^{n_0}_{R_{n_0}-1} \ni v_{-1}$ to return to $\cI^{n_0+1}_{(k+1)R_{n_0-1}}$. For $n > n_0$, the multiple projections (at various depths) can be applied to the orbit of $\cI^n_0$ near $v_0$ before they are undone near $v_{-1}$.}
\label{fig.hdiffeo}
\end{figure}

We also define another sequence of maps $\{\hH_i\}_{i=0}^{R_N-1}$ as follows (if $N=\infty$, then $R_N = \infty$). If $i < 2R_{n_0}$, let
$
\hH_i := F^i.
$
Otherwise, let $n_0 \leq n < N$ be the largest number such that $i \geq 2R_n$, and define
$
\hH_i := H^n_i.
$
Observe that by \lemref{lem unproj}, we have
\begin{equation}\label{eq.proj to original}
\hH_{R_n-1}|_{\cI^n_0} = H^{n-1}_{R_n-1}|_{\cI^n_0} = \cP^{n_0}_{-1}|_{\cI^n_{R_n-1}}\circ F^{R_n-1}|_{\cI^n_0}.
\end{equation}

\begin{rem}
In the definition of $\hH_i := H^n_i$, we set $n$ to be the largest number such that $i \geq 2 R_n$ rather than $i \geq R_n$ for the following technical reason. Observe that
$$
H^n_{R_n}(\cI^{n+1}_0) = \cP^n_0(\cI^{n+1}_{R_n}).
$$
The domain of $\cP^n_0$ is equal to $\cB^n_0$, whose right vertical boundary is $\lambda^{\bepsilon R_n}$ distance away from $\cI^{n+1}_{R_n}$. Hence, $H^n_{R_n}|_{\cI^{n+1}_0}$ does not extend to a horizontal curve $\ticI^{n+1}_{0, i}$ substantially larger than $\cI^{n+1}_0$ (so that its image would cover the adjacent neighbors of $H^n_{R_n}(\cI^{n+1}_0)$), since if it did, then $F^{R_n}(\ticI^{n+1}_{0, i})$ would lie outside of the domain $\cB^n_0$ of $\cP^n_0$.
\end{rem}

The remainder of the section is devoted to the proof of the following theorem, whose corollary immediately implies the Main Theorem.

\begin{thm}\label{a priori}
There exists a uniform constant
$$
\bfK = \bfK(L, \lambda, \epsilon, \lambda^{1-\epsilon}\|DF^{-1}\|, \|DF\|_{C^5}, \|F^{R_{n_0}}|_{\cB^{n_0}}\|_{C^6}, \kappa_F)\geq 1
$$
such that for all $n_0 \leq n \leq N$, we have
$$
\Dis(\hH_i, \cI^n_0) < \bfK
\matsp{for}
0\leq i < R_n.
$$
\end{thm}

\begin{cor}\label{a priori cor}
For $n_0 \leq n \leq N$, let $h_n : I^n_0 \to h_n(I^n_0)$ be the diffeomorphism given in \thmref{crit chart} ii). Then $\Dis(h_n, I^n_0) < \bfK$, where $\bfK > 1$ is the uniform constant given in \thmref{a priori}.
\end{cor}

Observe that any number $2R_{n_0} \leq i <R_N$ can be uniquely expressed as
$$
i = j + a_{n_0}R_{n_0} + a_{n_0+1} R_{n_0+1} + \ldots + a_nR_n
$$
for some $n_0 \leq n < N$, where
\begin{enumerate}[i)]
\item $0 \leq j < R_{n_0}$;
\item $0 \leq a_m < r_m$ for ${n_0} \leq m < n$; and
\item $2 \leq a_n < 2r_n$.
\end{enumerate}
In this case, we denote
$$
i := j + [a_{n_0}, a_{n_0+1}, \ldots, a_n].
$$
We extend this notation to $i < 2R_{n_0}$ by writing
$$
i = j + [a_{n_0}]
\matsp{for some}
a_{n_0} \in \{0, 1\}
$$

We record the following easy observation.

\begin{lem}\label{H decomp}
Let $2R_{n_0} \leq i <R_N$ be given by
$$
i = j + [a_{n_0}, \ldots, a_n].
$$
Then we have
$$
\hH_i = H^n_i = F^j\circ \left(\cP^{n_0}_0 \circ F^{a_{n_0}R_{n_0}}\right) \circ \ldots \circ \left(\cP^n_0 \circ F^{a_nR_n}\right).
$$
\end{lem}

For $n_0 \leq n \leq N$, we define a collection of arcs $\{\cJ^n_i\}_{i=0}^{R_n-1}$ by
\begin{equation}\label{eq.Js}
\cJ^n_i := \hH_i(\cI^n_0)
\matsp{for}
0 \leq i < R_n.
\end{equation}
See \figref{fig.Js}.

\begin{lem}\label{J spread}
Let $n_0 \leq n \leq N$ and $0\leq i < R_n$. If 
$$
i = [0, \ldots, 0, a_m, a_{m+1}, \ldots, a_k]
$$
for some ${n_0} \leq m \leq k < n$, then we have
$
\cJ^n_i \subset \cI^m_0.
$
Moreover, we have
$$
 \cJ^n_{i + l} = H^{m-1}_l(\cJ^n_{i})
 \matsp{for}
 0 \leq l < R_m.
 $$
\end{lem}

\begin{proof}
Observe that
$$
\cP^k_1 \circ F^{a_kR_k}(\cI^{k+1}_1) \subset \cI^k_1.
$$
By \lemref{H decomp}, the result follows from induction.
\end{proof}

\begin{figure}[h]
\centering
\includegraphics[scale=0.2]{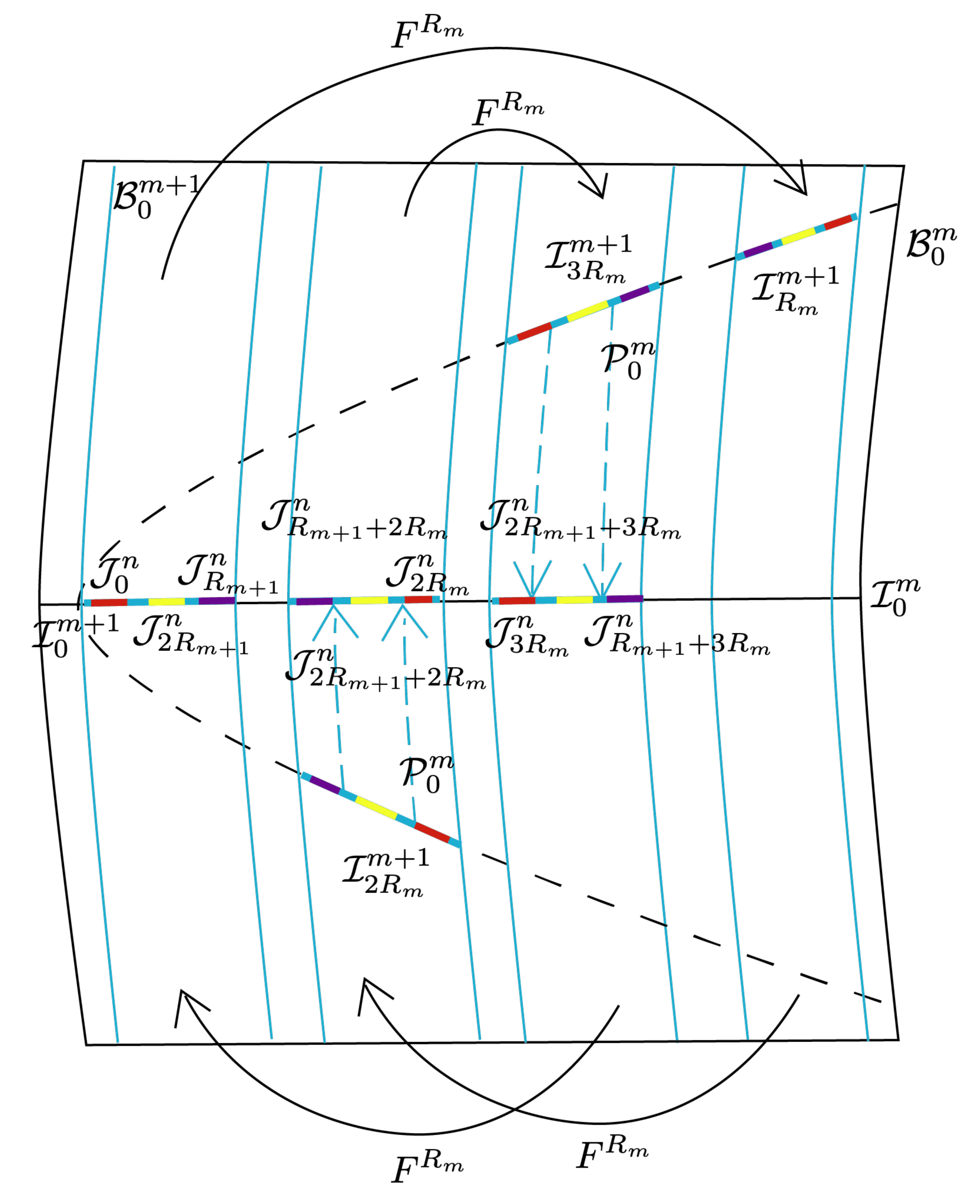}
\caption{Arcs $\cJ^n_i := \hH_i(\cI^n_0)$ with $0 \leq i <R_n$ that are contained in $\mathcal{I}^m_0$ for some $m < n$. For $0 \leq k < r_{m+1}$, we have $\cJ^n_{kR_{m+1}} \subset \cI^{m+1}_0$. For $2 \leq l < r_m$, we have $\cJ^n_{kR_{m+1}+lR_m}= \cP^m_0\circ F^{R_m}(\cJ^n_{kR_{m+1}})$.}
\label{fig.Js}
\end{figure}

\begin{lem}\label{on center track}
For $n_0 \leq n \leq N$ and $0 \leq i < R_n$, we have
$
\cJ^n_{i} \subset \cI^{n_0}_{i \md{R_{n_0}}}.
$
\end{lem}

\begin{proof}
The result follows immediately from \lemref{J spread}.
\end{proof}

Let $\Gamma : [0, 1] \to \bbR^2$ be a parameterized Jordan arc. For
$$
0 \leq a < b < c < d \leq 1,
$$
consider the subarcs $\Gamma_1 := \Gamma(a,b)$ and $\Gamma_2 := \Gamma(c,d)$ of $\Gamma$. We denote $\Gamma_1 <_\Gamma \Gamma_2.$ Let $\Gamma_3$ be another subarc of $\Gamma$. We denote
$
\Gamma_1 \leq_\Gamma \Gamma_3
$
if either
$
\Gamma_1 <_\Gamma \Gamma_3
$ or $
\Gamma_1 = \Gamma_3.$

Henceforth, we consider $\cI^{n_0}_0$ with parameterization given by
$$
\cI^{n_0}_0(t) := (\Psi^{n_0})^{-1}(t,0)
\matsp{for}
t \in I^{n_0}_0.
$$
Note that
$\cI^{n_0}_0 \circ P^{n_0}_0 = \cP^{n_0}_0.$
Moreover,
$$
P^{n_0}_0(v_0) = 0 < P^{n_0}_0(v_{R_{n_0}}).
$$

\begin{lem}\label{J order}
For $s \in \{1, 2\}$; $n_0 \leq n \leq N-s$ and $1 < k < R_{n+s}/R_n$, we have
$$
\cJ^{n+s}_0  <_{\cI^{n_0}_0} \cJ^{n+s}_{kR_n} <_{\cI^{n_0}_0}\cJ^{n+s}_{R_n}.
$$
\end{lem}

\begin{proof}
Observe that
\begin{itemize}
\item For $s\in\{1, 2\}$:
$$
\cJ^{n+s}_{R_n} = H^{n-1}_{R_n}(\cI^{n+s}_0) = \cP^{n-1}_0\circ F^{R_n}(\cI^{n+s}_0).
$$
\item For $1 < k < sr_n$:
$$
\cJ^{n+s}_{kR_n} = H^n_{kR_n}(\cI^{n+s}_0) = \cP^n_0\circ F^{kR_n}(\cI^{n+s}_0).
$$
\end{itemize}
In the case $s=1$, and the case $s=2$ and $1 < k < 2r_n$ follow immediately from \propref{value struct bound}.

Replacing $n$ by $n+1$ and applying the above conclusion, we see that for $1 < l < r_{n+1}$:
$$
\cJ^{n+2}_0 <_{\cI^{n_0}_0} \cJ^{n+2}_{lR_{n+1}} <_{\cI^{n_0}_0} \cJ^{n+2}_{R_{n+1}}.
$$
Note that for $2 < k < r_n$:
$$
\cJ^{n+2}_{lR_{n+1} + kR_n} = H^n_{kR_n}|_{\cI^{n+1}_0}(\cJ^{n+2}_{lR_{n+1}}).
$$
The result now follows from \lemref{H diffeo near tip}.
\end{proof}

Let $\Gamma_0 : [0, |\Gamma_0|] \to \bbR^2$ be a $C^1$-curve parameterized by its arclength. Let $\Gamma_1 := \Gamma_0(l, |\Gamma_0|-l)$ for some $0 < l < |\Gamma_0|/2$ be a subarc of $\Gamma_0$. We denote $\Gamma_1 = \Gamma_0[-l]$ and $\Gamma_0 = \Gamma_1[+l]$. Let $\Gamma_2$ be a $C^1$-curve such that $\Gamma_1 \subset \Gamma_2 \subset \Gamma_0$. We denote
$$
\Gamma_0\{-l\}= \Gamma_2 = \Gamma_1\{+l\}.
$$
Lastly, if $\Gamma_3$ and $\Gamma_4$ are $C^1$-curves in $\bbR^2$ and we have
$
\Gamma_3[-l] \subset \Gamma_4 \subset \Gamma_3[+l],
$
then we denote
$
\Gamma_4 = \Gamma_3\{\sim l\}.$ See \figref{fig.int}. These notations can be extended to intervals in $\bbR$ in the obvious way.

\begin{figure}[h]
\centering
\includegraphics[scale=0.25]{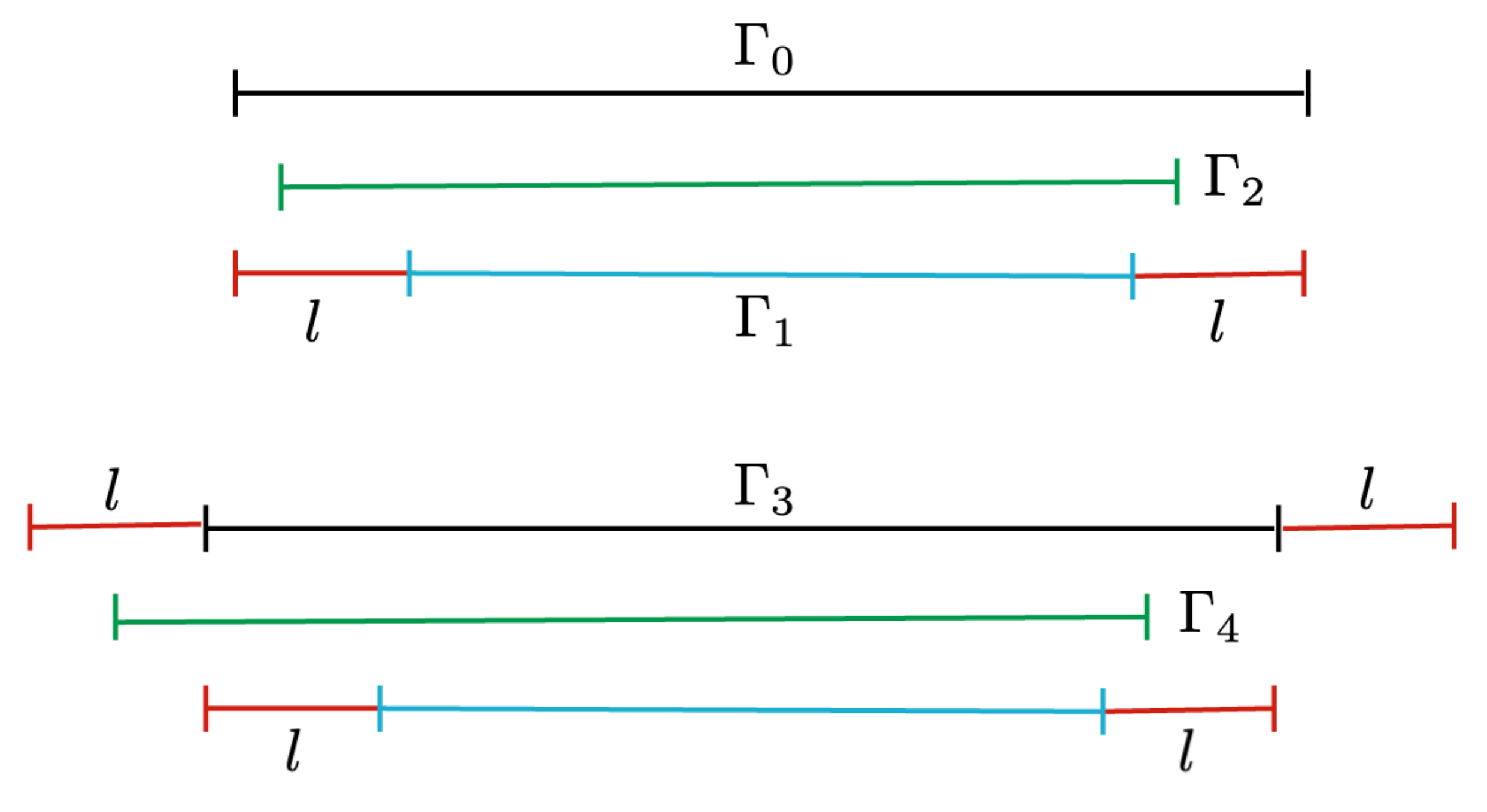}
\caption{Illustration of the relations between $\Gamma_0 = \Gamma_1[+l]$, $\Gamma_1 = \Gamma_0[-l]$, $\Gamma_0\{-l\} = \Gamma_2 = \Gamma_1\{+l\}$ (above); and $\Gamma_3$ and $\Gamma_4 = \Gamma_3\{\sim l\}$ (below).}
\label{fig.int}
\end{figure}

Let ${n_0} < n \leq N$, and consider the collection of arcs $\{\cJ^n_i\}_{i=0}^{R_n-1}$. By \lemref{on center track} and \lemref{J order}, for $2R_{n_0} \leq i < R_n$, there exist unique numbers $0 \leq \iota^n_-(i), \iota^n_+(i) < R_n$ such that
$$
\iota^n_\pm(i) = i \md{R_{n_0}},
$$
and the arcs $\cJ^n_{\iota^n_-(i)}$ and $\cJ^n_{\iota^n_+(i)}$ are the two nearest neighbors of $\cJ^n_i$ (one on each side) in $\cI^{n_0}_{i \md{R_{n_0}}}$. Define $\hcJ^n_i$ as the convex hull of $\cJ^n_{\iota^n_-(i)} \cup\cJ^n_i \cup \cJ^n_{\iota^n_+(i)}$ in $\cI^{n_0}_{i \md{R_{n_0}}}$.

We also define a subarc $\ticJ^n_i$ of $\cI^{n_0}_{i \md{R_{n_0}}}$ containing $\cJ^n_i$ as follows. Write
$$
i = j + [a_{n_0}, a_{n_0+1}, \ldots, a_m]
$$
for some ${n_0}\leq m < n$. If $m < n-1$, define
$$
\ticJ^n_i := \hcJ^n_i[+\lambda^{\bepsilon R_m}].
$$
Otherwise, define
$$
\ticJ^n_i := \hcJ^n_i[-\lambda^{\bepsilon R_{n-1}}].
$$

\begin{prop}\label{length bound}
There exists a uniform constant $K>0$ such that for $n_0 \leq n \leq N$, we have
$$
\sum_{i=2R_{n_0}}^{R_n-1} |\ticJ^n_i| < K.
$$
\end{prop}

\begin{proof}
Observe that
$$
\sum_{i=2R_{n_0}}^{R_n-1} |\ticJ^n_i| < \sum_{i=2R_{n_0}}^{R_n-1} |\hcJ^n_i| + \sum_{m={n_0}}^{n-1} 2R_{m+1}\lambda^{\bepsilon R_m}.
$$
By \lemref{J order}, the maximum number of overlaps among arcs in $\{\hcJ^n_i\}_{2R_{n_0}}^{R_n-1}$ is three. Hence, the above sum has a uniform upper bound.
\end{proof}

\begin{lem}\label{H length change}
For $n_0 \leq n \leq N$, let $\Gamma_0 \subset \cI^n_0$ be an arc. Then we have
$$
{K_0}^{-1}\lambda^{\bepsilon i} < \frac{\left|H^n_i(\Gamma_0)\right|}{\left|\Gamma_0\right|} < {K_0}\lambda^{-\bepsilon i}
\matsp{for}
0 \leq i < R_n,
$$
where $K_0 \geq 1$ is the uniform constant given in \eqref{eq.const 0}.
\end{lem}

\begin{proof}
For $p_0 \in \Gamma_0$, let $E_{p_0} \in \bbP^2_{p_0}$ be the direction tangent to $\Gamma_0$ at $p_0$. Note that $p_0$ is $R_n$-times forward $(L, \epsilon, \lambda)$-regular along $E^v_{p_0}$. Thus, by Theorems \ref{crit chart} and \ref{reg chart}, and \corref{ver hor cons}, we have
$$
{K_0}^{-1}\lambda^{\bepsilon l} < \|DF^l|_{E_{p_0}}\| < {K_0}\lambda^{-\bepsilon l}
\matsp{for}
0\leq l <R_n.
$$

By \propref{value struct bound} and \lemref{unproject} i), the curve
$
\Gamma_{kR_m} := F^{kR_m}(\Gamma_0)
$
for $0\leq k < r_m$ is $\lambda^{-\bepsilon R_m}$ horizontal in $\cB^m_0$. Hence, by \thmref{crit chart}, we see that
$$
{K_0}^{-1}\lambda^{\bepsilon R_m}< \|D\cP^m_0|_{E_{p_{kR_m}}}\| < {K_0}.
$$

Write
$$
i = j + [a_{n_0}, \ldots, a_m]
$$
for some ${n_0}\leq m < n$. Then by \lemref{H decomp} we have
$$
H^n_i = F^j\circ \cP^{n_0}_0 \circ F^{a_{n_0}R_{n_0}} \circ \ldots \circ \cP^m_0 \circ F^{a_mR_m}.
$$
Concatenating the previous estimates, we obtain the desired result.
\end{proof}

\begin{lem}\label{close pieces near crit}
For $s \in \{1, 2\}$; $n_0 \leq n \leq N-s$ and $2 \leq k < 2r_n$, let $X_{-1} \subset \cB^n_{R_n-1}$ be a set such that
$$
\cP^{n_0}_{-1}(X_{-1}) = \cJ^{n+s}_{kR_n-1}.
$$
Then
$$
\cP^n_0 \circ F(X_{-1}) =\cJ^{n+s}_{kR_n}\{\sim\lambda^{(1-\bepsilon) R_n}\}.
$$
\end{lem}

\begin{proof}
By \lemref{lem unproj}, we have
$$
\cI^{n+s}_{kR_n-1} = \left(\cP^{n_0}_{-1}|_{\cI^{n+s}_{kR_n-1}}\right)^{-1}(\cJ^{n+s}_{kR_n-1}) =  \left(\cP^{n_0}_{-1}|_{\cI^{n+s}_{kR_n-1}}\right)^{-1}\circ\cP^{n_0}_{-1}(X_{-1}).
$$
Since
$$
\cJ^{n+s}_{kR_n} = \cP^n_0 \circ F(\cI^{n+s}_{kR_n-1}),
$$
the claim follows from \eqref{eq.henon trans} and \eqref{eq.first entry squeeze}.
\end{proof}

\begin{prop}\label{cK}
For $n_0 \leq n \leq N-2$ and $2R_n \leq i < 2R_{n+1}$, there exists an arc $\cK_{0,i}$ containing $\cI^{n+2}_0$ such that the following properties are satisfied.
\begin{enumerate}[i)]
\item We have $\cK_{0, i} \supset \cK_{0, i+1}$.
\item The map $\hH_i|_{\cK_{0, i}}$ is a diffeomorphism.
\item We have $\hH_i(\cK_{0, i}) \supset \ticJ^{n+1}_{ i}$.
\item Denote
$
\cK_i := F^i(\cK_{0, i}).
$
Then for $2 < k \leq 2r_n$, the arc $\cK_{kR_n-1}$ is $\lambda^{(1-\bepsilon)R_n}$-horizontal in $\cB_{-1}$, and 
$$
\cK_{kR_n} \subset \cB^n_{R_n} \setminus \cV_{v_0}(\lambda^{\bepsilon R_n}).
$$
\end{enumerate}
\end{prop}

\begin{proof}
We first extend $\cI^{n_0+1}_{2{n_0}-1}$ to an arc $\cK_{2R_{n_0}-1} \subset \cB_{-1}$ such that $\cK_{2R_{n_0}-1}$ is  $\lambda^{(1-\bepsilon)R_{n_0}}$-horizontal in $\cB_{-1}$, and the curve
$
\cK_{2R_{n_0}} := F(\cK_{2R_{n_0}-1})
$
maps diffeomorphically onto
$
\cI^{n_0}_0 \setminus \cV_{v_0}(\lambda^{\bepsilon R_{n_0}})
$
under
$\cP^{n_0}_0|_{\cK_{2R_{n_0}}}$. We define
$$
\cK_{0, 2R_{n_0}} := F^{-2R_{n_0}}(\cK_{2R_{n_0}}).
$$

Proceeding by induction, suppose the result holds for $i \leq (k-1)R_n$ with $2 < k \leq 2r_n$. For $0 \leq l < R_n$, define
$$
\cK_{0, (k-1)R_n+l} := \cK_{0, (k-1)R_n}.
$$
Observe that
$$
\hH_{(k-1)R_n+l} = H^n_l \circ F^{(k-1)R_n}.
$$
Thus, property ii) follows from \lemref{H diffeo near tip}; property iii) follows from Lemmas \ref{J spread} and \ref{H length change}; and property iv) for $\cK_{kR_n-1}$ follows from \lemref{unproject} ii).

If $k < 2r_n$, then define $\cK_{kR_n}$ to be the component of
$
F(\cK_{kR_n-1}) \setminus \cV_{v_0}(\lambda^{\bepsilon R_n})
$
containing $\cI^{n+2}_{kR_n}$. By \lemref{unproject} i), $\cK_{kR_n}$ maps injectively under $\cP^n_0$. Lastly, property iii) follows from \lemref{close pieces near crit}.

If $k = 2r_n$, then define $\cK_{2R_{n+1}}$ to be the component of
$$
F(\cK_{2R_{n+1}-1})  \cap \left(\cB^{n+1}_0\setminus \cV_{v_0}(\lambda^{\bepsilon R_{n+1}})\right)
$$
containing $\cI^{n+3}_{2R_{n+1}}$. Properties ii) and iii) for $\cK_{2R_{n+1}}$ can be checked similarly as above.
\end{proof}

By \lemref{J order}, for $n_0 \leq n \leq N-2$, there exists a unique number $2 \leq \chi_n < r_n$ such that
$$
\cJ^{n+1}_0 <_{\cI^{n_0}_0} \cJ^{n+1}_{\chi_nR_n} \leq_{\cI^{n_0}_0} \cJ^{n+1}_{kR_n} 
\matsp{for all}
1\leq k < r_n.
$$
After relabelling $\iota^n_\pm$ if necessary, the following results hold.

\begin{lem}\label{real neigh}
Let $n_0 \leq n \leq N-2$. Then
$$
\iota^{n+1}_+(i) = i+\chi_nR_n
\matsp{for}
2R_{n_0} \leq i < R_n.
$$
\end{lem}

\begin{proof}
The claim follows immediately from Lemmas \ref{H diffeo near tip} and \ref{J spread}.
\end{proof}

\begin{lem}\label{fake neigh}
Let ${n_0}+2 \leq n \leq N$. For $1 \leq m \leq n-2$ and $2 \leq k < 2r_m$, there exists $1 \leq i < 2r_m$ such that
$$
\iota_-^n(kR_m)= \iota_-^{m+2}(kR_m) = iR_m.
$$
\end{lem}

\begin{proof}
By Lemmas \ref{J order}, \ref{H diffeo near tip} and \ref{J spread}, we see that the extremal intervals in $\cJ^{m+1}_{lR_m}$ for $0\leq l< r_m$ are $\cJ^n_{lR_m}$ and $\cJ^n_{lR_m+R_{m+1}}$. Moreover, by \lemref{real neigh}, we have
$$
\cJ^n_{\iota_+^n(lR_m + jR_{m+1})} \subset \cJ^{m+1}_{lR_m}
\matsp{for}
j \in \{0, 1\}.
$$
The claim follows.
\end{proof}

\begin{prop}\label{ticI}
For ${n_0}+2 \leq n \leq N$ and $2R_{n_0}\leq i < R_n$, there exists an arc $\ticI^n_{0,i}$ such that the following conditions hold for all $2R_{n_0} \leq j \leq i$.
\begin{enumerate}[i)]
\item We have $\cI^n_0 \subset \ticI^n_{0,i} \subset \cK_{0, i}$.
\item Denote
$$
\ticJ^n_{j, i-j} := \hH_j(\ticI^n_{0,i}).
$$
Then we have
$$
\ticJ^n_{j, i-j} \subset \ticJ^n_j
\matsp{and}
\ticJ^n_{i, 0} \supset \ticJ^n_i.
$$
\end{enumerate}
\end{prop}

\begin{proof}
First consider the case when $i <2R_{n-1}$. Proceeding by induction, suppose that the result is true for $j \leq kR_m$ with ${n_0} \leq m \leq n-2$ and $2 \leq k < 2r_m$. Then the result holds for $kR_m < j < (k+1)R_m$ by Lemmas \ref{H diffeo near tip} and \ref{J spread}.

Note that we have,
$$
\cP^m_0(\cK_{kR_m}) \supset \ticJ^{m+2}_{kR_m}\supset \cJ^{m+2}_{\iota^{m+2}_-(kR_m)}\cup \cJ^{m+2}_{kR_m}\cup \cJ^{m+2}_{\iota^{m+2}_+(kR_m)},
$$
where by Lemmas \ref{real neigh} and \ref{fake neigh}, we have
$$
\cJ^{m+2}_{\iota^{m+2}_-(kR_m)} = \cJ^{m+2}_{\iota^n_-(kR_m)} \supset \cJ^n_{\iota^n_-(kR_m)}
\matsp{and}
\cJ^{m+2}_{kR_m} \supset \cJ^n_{kR_m} \cup \cJ^n_{\iota^n_+(kR_m)}.
$$
Hence, there exists an arc $\cI_{kR_m}' \subset \cK_{kR_m}$ such that
$$
\cP^m_0(\cI_{kR_m}') = \ticJ^{m+2}_{kR_m}.
$$
By Lemmas \ref{H length change} and \ref{lem unproj}, we have
$$
\cP^{n_0}_{-1} \circ F^{R_m-1}(\cI_{kR_m}') = \hcJ^{m+2}_{(k+1)R_m-1}[+\lambda^{\bepsilon R_m}].
$$
Thus, by Lemmas \ref{close pieces near crit} and \ref{real neigh}, we see that 
$$
\cP^m_0 \circ F^{R_m}(\cI_{kR_m}') \supset \hcJ^{m+2}_{(k+1)R_m},
$$
and hence, the result holds for $j = (k+1)R_m$.

Next, consider the case when $i \geq 2R_{n-1}$. For $j < 2R_{n-1}$, the result follows by the same argument as in the previous case. Proceeding by induction, suppose that the result is true for $j \leq kR_{n-1}$ with $2 \leq k < r_{n-1}$. Then the result holds for $kR_{n-1} < j < (k+1)R_{n-1}$ by Lemmas \ref{H diffeo near tip}, \ref{J spread} and \lemref{H length change}.

Similar to the previous case, there exists an arc $\cI_{kR_{n-1}}' \subset \cK_{kR_{n-1}}$ such that
$$
\cP^{n-1}_0(\cI_{kR_{n-1}}') \supset \ticJ^n_{kR_{n-1}}
$$
and
$$
\cP^{n_0}_{-1} \circ F^{R_{n-1}-1}(\cI_{kR_{n-1}}') = \hcJ^{m+2}_{(k+1)R_{n-1}-1}[-\lambda^{\bepsilon R_n}].
$$
Let $\cI_{(k+1)R_{n-1}}''$ be the connected component of
$$
F(\cI_{(k+1)R_{n-1}}') \setminus \cV_{v_0}(\lambda^{\bepsilon R_n})
$$
containing $\cI^n_{(k+1)R_{n-1}}$. By \lemref{close pieces near crit}, we have
$$
\cP^{n-1}_0 (\cI_{(k+1)R_{n-1}}'') \supset \hcJ^n_{(k+1)R_{n-1}}[-\lambda^{\bepsilon R_n}].
$$
Thus, the result holds for $j = (k+1)R_{n-1}$.
\end{proof}

Let $i \geq 2R_{n_0}$ be a number given by
$$
i = [0, \ldots, 0, a_m, a_{m+1}, \ldots, a_k]
$$
for some ${n_0} \leq m \leq k$ so that $a_m > 0$. Denote
$$
\hm(i) := m,
\hspace{5mm}
\hk(i) := k
\matsp{and}
\ha(i) := a_m.
$$
We extend this notation to the case when
$
i = a_{n_0}R_{n_0}
$ with $
a_{n_0} \in \{0, 1\}
$
by letting
$$
\hm(i) := 1,
\hspace{5mm}
\hk(i) := 1
\matsp{and}
\ha(i) := a_{n_0}.
$$

\begin{prop}\label{H decomp fine}
Let $n_0 \leq n \leq N$ and $i = j+sR_{n_0}$ with $0\leq j < R_{n_0}$ and $0 \leq s <R_n/R_{n_0}$. For $0\leq l \leq s$, denote
$$
\hm_l := \hm(lR_{n_0}),
\hspace{5mm}
\hk_l := \hk(lR_{n_0})
\matsp{and}
\ha_l := \ha(lR_{n_0}).
$$
If $\hm_l = \hk_l$, let
$$
\chcI^n_l := F^{lR_{n_0}-1}(\ticI^n_{0, i}).
$$
Otherwise, let
$$
\chcI^n_l := \cI^{\hm_l+1}_{\ha_lR_{\hm_l}-1}.
$$
Then $\chcI^n_l$ is $\lambda^{(1-\bepsilon)R_{\hm_l}}$-horizontal. Moreover, define
$$
\chH_l :=  \cP^{\hm_l}_0 \circ F \circ \left(\cP^{n_0}_{-1}|_{\chcI^n_l}\right)^{-1} \circ F^{R_{n_0}-1}|_{\cI^{n_0}_0}.
$$
Then we have
$$
\hH_i|_{\ticI^n_{0, i}} = F^j|_{\cI^{n_0}_0} \circ \chH_s \circ \ldots\circ \chH_4 \circ \chH_3 \circ \cP^{n_0}_0 \circ F^{2R_{n_0}}|_{\ticI^n_{0, i}}.
$$
\end{prop}

\begin{proof}
We proceed by induction. Clearly, the result is true for $i < 2R_{n_0}$. Suppose that the result is true for all $i' < i$.

First, suppose $i = 2R_{k+1}$ for some ${n_0} \leq k+1 < n$. Denote
$$
\Gamma_d := F^d(\ticI^n_{0, i})
\matsp{for}
0\leq d \leq i.
$$
By \lemref{H decomp}:
\begin{equation}\label{eq.decomp 1}
\hH_{2R_{k+1}}|_{\Gamma_0} = \cP^{k+1}_0 \circ F^{2R_{k+1}} = \cP^{k+1}_0 \circ F \circ F^{R_k-1} \circ F^{(2r_k-1)R_k}|_{\Gamma_0}.
\end{equation}

By \propref{cK} iv), $\Gamma_{(2r_k-1)R_k}$ is $\lambda^{-\bepsilon R_k}$-horizontal in $\cB^k_0$. So it follows from \lemref{first entry hor graph} that
$$
F^{R_k-1}|_{\Gamma_{(2r_k-1)R_k}}= \left(\cP^{n_0}_{-1}|_{\Gamma_{2R_{k+1}-1}}\right)^{-1} \circ F^{R_k-1} \circ  \cP^k_0|_{\Gamma_{(2r_k-1)R_k}}.
$$
Note that
$$
\hH_{(2r_k-1)R_k} = H^k_{(2r_k-1)R_k} = \cP^k_0 \circ F^{(2r_k-1)R_k}.
$$
Substituting into \eqref{eq.decomp 1}, we obtain
$$
\hH_{2R_{k+1}}|_{\Gamma_0} =\cP^{k+1}_0 \circ F \circ \left(\cP^{n_0}_{-1}|_{\Gamma_{2R_{k+1}-1}}\right)^{-1} \circ F^{R_k-1} \circ \hH_{(2r_k-1)R_k}|_{\Gamma_0}.
$$
By \lemref{lem unproj}, we have
$$
F^{R_k-1}|_{\cI^k_0} =\left(\cP^{n_0}_{-1}|_{\cI^k_{R_k-1}}\right)^{-1} \circ H^k_{R_k-1}|_{\cI^k_0}.
$$
Thus, we conclude:
$$
\hH_{2R_{k+1}}|_{\Gamma_0}=\cP^{k+1}_0 \circ F \circ \left(\cP^{n_0}_{-1}|_{\Gamma_{2R_{k+1}-1}}\right)^{-1} \circ H^k_{R_k-1}|_{\cI^k_0} \circ \hH_{(2r_k-1)R_k}|_{\Gamma_0}.
$$

We can apply the induction hypothesis to decompose $\hH_{(2r_k-1)R_k}$ into factors of the form $\chH_l$. Observe that for
$$
e_0:= (2r_k-1)R_k < e < 2R_{k+1},
$$
we have
$$
\hm(e) = \hm(e - e_0)< \hk(e) \leq k
\matsp{and}
\ha(e) = \ha(e-e_0).
$$
Hence, we can also apply the induction hypothesis to $H^k_{R_k-1}|_{\cI^k_1}$ to decompose them into factors of the form $\chH_l$. The claim follows.

Next, suppose that
$
i = a_kR_k
$
for some ${n_0} \leq k < n$ and $a_k \geq 3$. Proceeding in the same way as in the previous case, we obtain (in place of \eqref{eq.decomp 1}):
$$
\hH_i|_{\Gamma_0} = \cP^k_0 \circ F^{a_kR_k} = \cP^k_0 \circ F \circ F^{R_k-1}\circ F^{(a_k-1)R_k}|_{\Gamma_0}.
$$
The rest of the argument is identical {\it mutatis mutandis}.

Lastly, suppose that
$$
i = j + [a_{n_0}, \ldots, a_k]
$$
for some ${n_0}< k < n$ such that
$$
\hm(i) < k = \hk(i) < n.
$$
Then
$$
\hH_i =  H^{k-1}_{i-a_kR_k} \circ \cP^k_0 \circ F^{a_k R_k} = H^{k-1}_{i-a_kR_k}|_{\cI^k_0} \circ \hH_{a_kR_k}.
$$
Applying the induction hypothesis to $\hH_{a_kR_k}$ and $H^{k-1}_{i-a_kR_k}|_{\cI^k_0}$ and arguing as above, we obtain the desired result.
\end{proof}

Let $G : U \to G(U)$ be a $C^1$-diffeomorphism defined on a domain $U \subset \bbR^2$. For a $C^1$-curve $\Gamma \subset U$, we define the {\it cross-ratio distortion $\CrD(G, \Gamma)$ of $G$ on $\Gamma$} as the cross-ratio distortion of 
$$
G_\Gamma := \phi_{G(\Gamma)}^{-1}\circ G \circ \phi_\Gamma,
$$
where
$
\phi_\Gamma
$ and $
\phi_{G(\Gamma)}
$
are parameterizations of $\Gamma$ and $G(\Gamma)$ by their respective arclengths (see \secref{sec.koebe}).

\begin{prop}\label{bound crd}
Let $n_0 \leq n \leq N$ and $1 \leq i < R_n$. Then there exists a uniform constant $\nu > 0$ such that the maps $\hH_i$ and $\hH_{R_n-1}\circ \hH_i^{-1}$ have $\nu$-bounded cross-ratio distortion on $\ticI^n_{0, i}$ and $\hH_i(\ticI^n_{0, R_n-1})$ respectively.
\end{prop}

\begin{proof}
Consider the decomposition of $\hH_i$ given in \propref{H decomp fine}:
$$
\hH_i|_{\ticI^n_{0, i}} = F^j|_{\cI^{n_0}_0} \circ \chH_s \circ \ldots \circ \chH_3 \circ \cP^{n_0}_0 \circ F^{2R_{n_0}}|_{\ticI^n_{0, i}}.
$$
Denote
$$
\cJ := \cP^{n_0}_0 \circ F^{2R_{n_0}}(\ticI^n_{0,i})
\matsp{and}
\chH := \chH_s \circ \ldots \circ \chH_3.
$$
To prove the cross-ratio distortion bound for $\hH_i$, it suffices to prove it for $\chH$ on $\cJ$.

The maps
$$
\phi_0 := (P^{n_0}_0|_{\cI^{n_0}_0})^{-1} : I^{n_0}_0 \to \cI^{n_0}_0
\matsp{and}
\phi_{-1} := (P_{-1}|_{\cI^{n_0}_{R_{n_0}-1}})^{-1} : I^{n_0}_{R_{n_0}-1} \to \cI^{n_0}_{R_{n_0}-1}
$$
give parameterizations of $\cI^{n_0}_0$ and $\cI^{n_0}_{R_{n_0}-1}$. Denote
$$
J_2 := \phi_0^{-1}(\cJ)
\matsp{and}
h_1 := \phi_{-1}^{-1} \circ F^{R_{n_0}-1}|_{\cI^{n_0}_0} \circ \phi_0.
$$

For $3 \leq l \leq s$, let
$$
H_l := \phi_0^{-1} \circ \chH_l \circ \ldots \circ \chH_3 \circ \phi_0;
$$
and
$$
J_l' := h_1(J_{l-1})
\matsp{and}
J_l := H_l(J_2).
$$
By \propref{H decomp fine} and \lemref{quad mapping}, there exist a diffeomorphism $\psi_l : J_l' \to \psi_l(J_l')$ and a constant $a_l \in \bbR$ such that 
$$
H_l(x) = (\psi_l \circ h_1 \circ H_{l-1}(x))^2+a_l.
$$

By \eqref{eq.crd product} and \lemref{power crd}, we see that
$$
\CrD(\chH, \cJ) \asymp \CrD(H_s, J_2) > \left(\prod_{l=2}^{s-1} \CrD(h_1, J_l)\right) \cdot \left(\prod_{l=3}^s \CrD(\psi_l, J_l')\right).
$$
Note that the diffeomorphisms $h_1$ and $\{\psi_l\}_{l=3}^s$ have uniformly bounded second derivatives. Moreover, Propositions \ref{length bound} and \ref{ticI} implies that the total length of $\{J_l, J_l'\}_{l=3}^s$ is uniformly bounded. The bound on the cross ratio distortion of $\hH_i$ now follows from \lemref{d crd}.

Now, consider the decomposition of $\hH_{R_n-1}$ on $\ticI^n_{0, R_n-1}$:
$$
\hH_{R_n-1}|_{\ticI^n_{0, R_n-1}} =F^{R_{n_0}-1}|_{\cI^{n_0}_0} \circ \chH_S \circ \ldots \circ \chH_3 \circ \cP^{n_0}_0 \circ F^{2R_{n_0}}|_{\ticI^n_{0, R_n-1}},
$$
where $S := R_n/R_{n_0}-1$. The same argument as above implies the bound on the cross ratio distortion of
$$
\hH_{R_n-1} \circ \hH_i^{-1}|_{\cI} = F^{R_{n_0}-1}|_{\cI^{n_0}_0} \circ \chH_S \circ \ldots \circ \chH_{S-s} \circ F^{R_{n_0}-1-j}|_{\cI}
$$
on $\cI := \hH_i(\ticI^n_{0, R_n-1})$.
\end{proof}

\begin{proof}[Proof of \thmref{a priori}]
Consider the arcs $\{\cJ^n_i\}_{i=0}^{R_n-1}$. There exists $2R_{n_0}\leq i_1 <R_n$ such that
$$
|\cJ^n_{\iota^n_+(i_1)}|\; ,\; |\cJ^n_{\iota^n_-(i_1)}| > k|\cJ^n_{i_1}|
$$
for some uniform constant $k >0$. By \propref{ticI}, there exists an arc $\ticI^n_{0, i_1} \supset \cI^n_0$ which is mapped diffeomorphically onto $\ticJ^n_{i_1}$ by $\hH_{i_1}$.

Recall that the nearest neighbor of $\cI^n_0$ in $\cI^{n_0}_0$ is given by $\cJ^n_{\chi_{n-1}R_{n-1}}$. Let $\hcI^n_0$ be the convex hull of $\cI^n_0 \cup \cJ^n_{\chi_{n-1}R_{n-1}}$. Then
$$
(\ticI^n_{0, i_1} \cap \cI^{n_0}_0) \setminus \cI^n_0 \subset \hcI^n_0 \setminus \cI^n_0.
$$
Hence, \propref{bound crd} and \thmref{koebe} imply
$$
\left|\hcI^n_0 \setminus \cI^n_0\right| > k \left| \cI^n_0\right|.
$$
By \lemref{close pieces near crit}, we conclude that the two components of $\ticJ^n_{R_n-1} \setminus \cJ^n_{R_n-1}$ have lengths greater than $k\left|\cJ^n_{R_n-1}\right|$. By \propref{ticI}, $\hH_{R_n-1}$ maps $\ticI^n_{0, R_n-1} \supset \cI^n_0$ diffeomorphically onto $\ticJ^n_{R_n-1}$. The result now follows from \propref{bound crd} and \thmref{koebe}.
\end{proof}


\appendix


\section{Quantitative Pesin Theory}\label{sec.pesin}

In this section, we summarize the results in \cite{CLPY1}. Let $r \geq 2$ be an integer, and consider a $C^{r+1}$-diffeomorphism $F : \cB \to F(\cB) \Subset \cB$, where $\cB \subset \bbR^2$ is a bounded domain. Let $\lambda, \epsilon \in (0,1)$ with $\bepsilon < 1$.

Let $p_0 \in \cB$ and $E^v_{p_0} \in \bbP^2_{p_0}$. For $m \in \bbZ$, decompose the tangent space at $p_m$ as
$$
\bbP^2_{p_m} = (E^v_{p_m})^\perp \oplus E^v_{p_m}.
$$
In this decomposition, we have
$$
D_{p_m}F =: \begin{bmatrix}
\alpha_m & 0\\
\zeta_m & \beta_m
\end{bmatrix},
$$
where $\alpha_m, \beta_m >0$ and $\zeta_m \in \bbR$.

For some $M, N \in \bbN \cup \{0, \infty\}$ and $L \geq 1$, suppose for $s \in \{0, 1\}$, we have
$$
L^{-1}\lambda^{(1+\epsilon)n} \leq \frac{\beta_{0} \ldots \beta_{{n-1}}}{(\alpha_0 \ldots \alpha_{n-1})^s}\leq L\lambda^{(1-\epsilon)n}
\matsp{for}
1 \leq n \leq N,
$$
and
$$
L^{-1}\lambda^{(1+\epsilon)n}\leq \frac{\beta_{{-n}} \ldots \beta_{{-1}}}{(\alpha_{{-n}} \ldots \alpha_{{-1}})^s}
\leq L\lambda^{(1-\epsilon)n}
\matsp{for}
1 \leq n \leq M.
$$
Then we say that $p_0$ is {\it $(M, N)$-times $(L, \epsilon, \lambda)$-regular along $E^v_{p_0}$}.

\begin{prop}[Growth in irregularity]\cite[Proposition 5.5]{CLPY1}\label{grow irreg}
For $-M \leq m \leq N$, let $\cL_{p_m} \geq 1$ be the minimum value such that $p_m$ is $(M+m, N-m)$-times $(\cL_{p_m}, \epsilon, \lambda)$-regular along $E^v_{p_m}$. Then
$$
\cL_{p_m} < \bL\lambda^{-\bepsilon |m|}.
$$
\end{prop}

\subsection{Linearization}

For $w, l > 0$, denote
$$
\bbB(w, l) := (-w, w)\times (-l, l) \subset \bbR^2
\matsp{and}
\bbB(l) := \bbB(l, l).
$$

\begin{thm}[Regular charts]\cite[Theorem 6.1]{CLPY2}\label{reg chart}
There exists a uniform constant
$$
C = C(\|DF\|_{C^r}, \lambda^{-\epsilon}) \geq 1
$$
such that the following holds. For $-M \leq m \leq N$, let
$$
\omega := \frac{\lambda^{1-\epsilon}}{1-\lambda^{1-\epsilon}}\cdot \|DF^{-1}\|\cdot\|DF\|
\matsp{and}
\cK_{p_m} := \bL(1+\omega)^5\|DF^{-1}\|\lambda^{1-\bepsilon |m|}.
$$
Define
$$
U_{p_m} := \bbB(l_{p_m})
\matsp{where}
l_{p_m} := \lambda^{1+\bepsilon}(C\cK_{p_m})^{-1}.
$$
Then there exists a $C^r$-chart $\Phi_{p_m} : (\cU_{p_m}, p_m) \to (U_{p_m}, 0)$ such that $D\Phi_{p_m}(E^v_{p_m}) = E^{gv}_0$, 
$$
\|D\Phi_{p_m}^{-1}\|_{C^{r-1}} < C(1+\omega)
\comma
\|D\Phi_{p_m}\|_{C^s} < C\cK_{p_m}^{s+1}
\matsp{for}
0 \leq s <r,
$$
and the map $\Phi_{p_{m+1}} \circ F|_{\cU_{p_m}} \circ \Phi_{p_m}^{-1}$ extends to a globally defined $C^r$-diffeomorphism $F_{p_m} : (\bbR^2, 0) \to (\bbR^2, 0)$ satisfying the following properties:
\begin{enumerate}[i)]
\item $\displaystyle \|DF_{p_m}\|_{C^{r-1}} \leq \|DF\|_{C^r}$;
\item we have
$$
D_0F_{p_m} =\begin{bmatrix}
a_m & 0\\
0 & b_m
\end{bmatrix},
\matsp{where}
\lambda^{\bepsilon} < a_m < \lambda^{-\bepsilon}
\matsp{and}
\lambda^{1+\bepsilon} < b_m < \lambda^{1-\bepsilon}.
$$
\item $\|D_z F_{p_m} - D_0F_{p_m}\|_{C^0} < \lambda^{1+\bepsilon}$ for $z \in \bbR^2$;
\item we have
$$
F_{p_m}(x,y) = (f_{p_m}(x), e_{p_m}(x,y))
\matsp{for}
(x,y) \in \bbR^2,
$$
where $f_{p_m}:(\bbR, 0) \to (\bbR, 0)$ is a $C^r$-diffeomorphism, and $e_{p_m} : \bbR^2 \to \bbR$ is a $C^r$-map such that for all $0 \leq s \leq r$, we have
$$
\partial_x^s e_{p_m}(\cdot, y) \leq \|DF\|_{C^r} |y|
\matsp{for}
y \in \bbR.
$$
\end{enumerate}
\end{thm}

The construction in \thmref{reg chart} is referred to as {\it a linearization of $F$ along the $(M,N)$-orbit of $p_0$ with vertical direction $E^v_{p_0}$}. For $-M \leq m \leq N$, we refer to $l_{p_m}$, $\cU_{p_m}$, $\Phi_{p_m}$ and $F_{p_m}$ as a {\it regular radius}, a {\it regular neighborhood}, a {\it regular chart} and a {\it linearized map at $p_m$} respectively.

For $p \in \bbR^2$ and $t > 0$, let
$$
\bbD_p(t) := \{\|q - p\| < t\}.
$$

\begin{lem}\cite[Lemma 6.2]{CLPY1}\label{size reg nbh}
For $-M \leq m \leq N$, we have
$$
\cU_{p_m} \supset \bbD_{p_m}\left(\frac{\lambda^{1+\bepsilon}}{C^2\cK_{p_m}^2}\right),
$$
where $C, \cK_{p_m} \geq 1$ are given in \thmref{reg chart}.
\end{lem}

\subsection{$C^1$-estimates}

\begin{prop}[Jacobian bounds]\cite[Proposition 6.14]{CLPY2}\label{jac bound}
We have
$$
\bL^{-1}\lambda^{(1+\bepsilon)n} \leq \Jac_{p_0}F^n \leq \bL\lambda^{(1-\bepsilon)n}
\matsp{for}
1 \leq n \leq N,
$$
and
$$
\bL^{-1}\lambda^{-(1-\bepsilon)n}\leq \Jac_{p_0}F^{-n} \leq \bL\lambda^{-(1+\bepsilon)n}
\matsp{for}
1 \leq n \leq M.
$$
\end{prop}

\begin{prop}[Derivative bounds]\cite[Proposition 6.15]{CLPY2}\label{ext deriv bound}
Let $C \geq 1$ and $\omega >0$ be the uniform constants given in \thmref{reg chart}. For $E_{p_0} \in \bbP^2_{p_0}$, we have
$$
\frac{\lambda^{(1+\bepsilon)n}}{C\bL(1+\omega)^2} \leq \|DF^n|_{E_{p_0}}\| \leq C(1+\omega)^2\lambda^{-\bepsilon n}
\matsp{for}
1 \leq n \leq N,
$$
and
$$
\frac{\lambda^{\bepsilon n}}{{C\bL}(1+\omega)^2}\leq \|DF^{-n}|_{E_{p_0}}\| \leq C(1+\omega)^2\lambda^{-(1+\bepsilon)n}
\matsp{for}
1 \leq n \leq M.
$$
\end{prop}

Consider the sequence of linearized maps $\{F_{p_m}\}_{-M}^N$ given in \thmref{reg chart}. For $1 \leq n \leq N-m$, we denote
\begin{equation}\label{eq.compose}
F_{p_m}^n = (f_{p_m}^n, e_{p_m}^n):= F_{p_{m +n-1}} \circ \ldots \circ F_{p_{m+1}}\circ F_{p_m}.
\end{equation}

\begin{prop}\cite[Proposition 6.4]{CLPY1}\label{lin comp}
For $-M \leq m \leq N$ and $0 \leq n \leq N-m$, consider the $C^r$-diffeomorphism $F_{p_m}^n$ given in \eqref{eq.compose}. Let $z = (x, y) \in U_{p_m}$, and suppose that
$$
z_i = (x_i, y_i) := F_{p_m}^i(z) \in U_{p_{m+i}}
\matsp{for}
0 \leq i \leq n.
$$
Denote
$$
D_zF_{p_m}^n =: \begin{bmatrix}
a_m^n(z) & 0\\
c_m^n(z) & b_m^n(z)
\end{bmatrix}.
$$
Define
$$
\chl_h := \sup_n n\lambda^{\bepsilon n} < \infty
\matsp{and}
\chl_v := (1-\lambda^{1-\bepsilon})^{-1};
$$
and
$$
\chi_h := \exp\left(\frac{\chl_h\|F\|_{C^3}}{\lambda^{\bepsilon}}\right)
\matsp{and}
\chi_v := \exp\left(\frac{(\chl_h+\chl_v)\|F\|_{C^3}}{\lambda^{\bepsilon}}\right).
$$
Then
$$
\frac{1}{\chi_h}\leq \frac{a_m^n(z)}{a_m^n(0)} \leq \chi_h
\comma
\frac{1}{\chi_v}\leq \frac{b_m^n(z)}{b_m^n(0)} \leq \chi_v
\matsp{and}
\|\gamma_m^n\| < \lambda^{(1-\bepsilon)n}.
$$
\end{prop}

For $-M \leq m \leq N$ and $q \in \cU_{p_m}$, write $z := \Phi_{p_m}(q)$. The {\it vertical/horizontal direction at $q$ in $\cU_{p_m}$} is defined as $E^{v/h}_q := D\Phi_{p_m}^{-1}(E^{gv/gh}_z)$. By the construction of regular charts in \thmref{reg chart}, vertical directions are invariant under $F$ (i.e. $DF(E_q^v) = E_{F(q)}^v$ for $q \in \cU_{p_m}$). Note that the same is not true for horizontal directions.

\begin{prop}\cite[Proposition 6.5]{CLPY2}\label{chart cons}
For $-M \leq m \leq N$ and $q \in \cU_{p_m}$, we have
$$
\frac{1}{\sqrt{2}} \leq \frac{\|D\Phi_{p_m}|_{E_z^{v/h}}\|}{\|D\Phi_{p_m}|_{E_{p_m}^{v/h}}\|} \leq \sqrt{2}.
$$
\end{prop}

\begin{cor}\cite[Corollary 6.6]{CLPY1}\label{ver hor cons}
For some $-M \leq m_0 \leq N$, let $q_{m_0} \in \cU_{p_{m_0}}$. Suppose for $m_0 \leq m \leq m_1 \leq N$, we have $q_m \in \cU_{p_m}$. Let
$$
\hE^h_{q_m} := DF^{m-m_0}(E^h_{q_{m_0}}).
$$
Then for $m_0 \leq m'\leq m_1$, we have
$$
\frac{1}{2\chi_h}\leq \frac{\|DF^{m'-m}|_{\hE^h_{q_m}}\|}{\|DF^{m'-m}|_{E^h_{p_m}}\|} \leq 2\chi_h
\matsp{and}
\frac{1}{2\chi_v}\leq \frac{\|DF^{m'-m}|_{E^v_{q_m}}\|}{\|DF^{m'-m}|_{E^v_{p_m}}\|} \leq 2\chi_v,
$$
where $\chi_h$ and $\chi_v$ are constants given in \propref{lin comp}.
\end{cor}

\begin{prop}[Vertical alignment of forward contracting directions]\cite[Proposition 6.8]{CLPY1}\label{vert angle shrink}
Let $q_0 \in \cU_{p_0}$ and $\tiE_{q_0}^v \in \bbP^2_{q_0}$. Suppose $q_i \in \cU_{p_ i}$ for $0 \leq i \leq n \leq N$, and that
$$
\nu := \|DF^n|_{\tiE_{q_0}^v}\| < \frac{\lambda^{\bepsilon n}}{\chi_h(2+\omega)^3 \bC},
$$
where $C, \omega, \chi_h \geq 1$ are uniform constants given in \thmref{reg chart} and \propref{lin comp}. Denote $z_0 := \Phi_{p_0}(q_0)$ and $\tiE_{z_0}^v := D\Phi_{p_0}(\tiE^v_{q_0})$. Then
$$
\measuredangle(\tiE^v_{z_0}, E_{z_0}^{gv}) < \chi_h (1+\omega)\bC\lambda^{-\bepsilon n} \nu.
$$
\end{prop}

\begin{prop}[Horizontal alignment of backward neutral directions]\cite[Proposition 6.9]{CLPY1}\label{hor angle shrink}
Let $q_0 \in \cU_{p_0}$ and $\tiE_{q_0}^h \in \bbP^2_{q_0}$. Suppose  $q_{-i} \in \cU_{p_ {-i}}$ for $0 \leq i \leq n \leq M$, and that
$$
\mu := \|DF^{-n}|_{\tiE^h_{q_0}}\| < \frac{1}{\chi_v(2+\omega)^3\bC \lambda^{(1-\bepsilon)n}}.
$$
Denote
$$
z_0 := \Phi_{p_0}(q_0)
\comma
\tiE_{z_0}^h := D\Phi_{p_0}(\tiE^h_{q_0})
\matsp{and}
\hE_{z_0}^h := D\Phi_{p_0} \circ F^n(E^h_{q_{-n}}).
$$
Then
$$
\measuredangle(\tiE_{q_0}^h, \hE_{q_0}^h) < \chi_v(1+\omega)\bC \lambda^{(1-\bepsilon)n}\cdot \mu.
$$
\end{prop}

The {\it $n$-times truncated regular neighborhood of $p_0$} is defined as
$$
\cU_{p_m}^n := \Phi_{p_m}^{-1}\left(U_{p_m}^n\right)\subset \cU_{p_m},
\matsp{where}
U_{p_0}^n := \bbB\left(\lambda^{\bepsilon n}l_{p_m}, l_{p_m}\right).
$$
The purpose of truncating a regular neighborhood is to ensure that its iterated images stay inside regular neighborhoods.

\begin{lem}\cite[Lemma 6.10]{CLPY1}\label{trunc neigh fit}
Let $-M \leq m \leq N$ and $0 \leq n \leq N-m$. We have $F^i(\cU_{p_m}^n) \subset \cU_{p_{m+i}}$ for $0 \leq i \leq n$.
\end{lem}

\begin{prop}\cite[Propositions B.5 and B.6]{CLPY1}\label{cr lin comp}
There exists a uniform constant $K = K(\|DF\|_{C^r}, \lambda, \epsilon, r) \geq 1$ such that the following result holds. For $-M \leq m \leq N$ and $0 \leq n \leq N-m$, consider the $C^r$-maps $f_{p_m}^n$ and $e_{p_m}^n$ given in \eqref{eq.compose}. Then we have
$$
\|Df_{p_m}^n\|_{C^{r-1}} < K\lambda^{-\bepsilon n}
\matsp{and}
\|De_{p_m}^n\|_{C^{r-1}} < K\lambda^{(1-\bepsilon)n}.
$$
\end{prop}

\subsection{$C^r$-estimates}

Let $g : \bbR \to \bbR$ be a $C^r$-function. The curve
$$
\Gamma_g := \{(x, g(x)) \; x \in \bbR\}
$$
is the {\it horizontal graph of $g$}. Let $H : \bbR^2 \to \bbR^2$ be a $C^r$-diffeomorphism. Suppose that there exists a $C^r$-function $H_*(g) : \bbR \to \bbR$ such that $H(\Gamma_g) = \Gamma_{H_*(g)}$. Then $H_*(g)$ and $\Gamma_{H_*(g)}$ are referred to as the {\it horizontal graph transform of $g$} and $\Gamma_g$ {\it by $H$} respectively.

\begin{prop}[$C^r$-convergence of horizontal graphs]\cite[Proposition 4.5]{CLPY1}\label{for gt}
Let $g : \bbR \to \bbR$ be a $C^r$-map with $\|g'\|_{C^{r-1}} < \infty$. For $-M \leq m \leq N$ and $1 \leq n \leq N-m$, consider the graph transform $\tig := (F_{p_m}^n)_*(g)$. Then
$$
\|\tig'\|_{C^{r-1}} < C\lambda^{(1-\bepsilon)n}(1+\|g'\|_{C^{r-1}})^r
$$
where $C = C(\bfC, \lambda,\epsilon, r) \geq 1$ is a uniform constant.
\end{prop}

For $p \in \bbR^2$ and $u \in \bbR$, let $E_p^u \in \bbP_p^2$ be the tangent direction at $p$ given by
$$
E_p^u := \{r(u, 1) \; | \; r \in \bbR\}.
$$
Let $\xi : \bbR^2 \to \bbR$ be a $C^{r-1}$-map. The direction field
$$
\cE_\xi := \{E_p^{\xi(p)} \; | \; p \in \bbR^2\}
$$
is the {\it vertical direction field of $\xi$}. Let $H : \bbR^2 \to \bbR^2$ be a $C^r$-diffeomorphism. Suppose that there exists a $C^{r-1}$-map $H^*(\xi) : \bbR^2 \to \bbR$ such that $DH^{-1}(\cE_\xi) =\cE_{H^*(\xi)}$. Then $H^*(\xi)$ and $\cE_{H^*(\xi)}$ are referred to as the {\it vertical direction field transform of $\xi$} and $\cE_\xi$ {\it by $H$} respectively.

\begin{prop}[Backward vertical direction field transform]\cite[Proposition 4.6]{CLPY1}\label{back dt}
There exist uniform constants $C, \tiC \geq 1$ depending only on $\bfC, \lambda, \epsilon, r$  such that the following holds. Let $\xi : \bbR^2 \to \bbR$ be a $C^{r-1}$-map with $\|\xi\|_{C^{r-1}} < \infty$. For $-M \leq m < N$ and $0 \leq n \leq M+m$, consider the vertical direction transform 
$$
\tixi := (F_{p_m}^n)^*(\xi)|_{\bbR \times (-1, 1)}.
$$
Suppose
$$
C\lambda^{(1-\bepsilon)n}(1+\|\xi\|_{C^{r-1}}) < 1.
$$
Then
$$
\|\tixi\|_{C^{r-1}} < \tiC\lambda^{(1-\bepsilon)n}\|\xi\|_{C^{r-1}}.
$$
\end{prop}

\subsection{Stable and center manifolds}

For $-M \leq m \leq N$, define the {\it local vertical} and {\it horizontal manifold at $p_m$} as
$$
W^v_{\loc}(p_m) := \Phi_{p_m}^{-1}(\{(0, y) \in U_{p_m}\})
\matsp{and}
W^h_{\loc}(p_m) := \Phi_{p_m}^{-1}(\{(x,0) \in U_{p_m}\})
$$
respectively.

If $N = \infty$, then \propref{vert angle shrink} implies that $E_{p_0}^v$ is the unique direction along which $p_0$ is infinitely forward regular. In this case, we denote $E_{p_0}^{ss} := E_{p_0}^v$, and refer to this direction as the {\it strong stable direction at $p_0$}. Additionally, we define the {\it strong stable manifold of $p_0$} as
$$
W^{ss}(p_0) := \left\{q_0 \in \Omega \; | \; \limsup_{n \to \infty}\frac{1}{n}\log\|q_n - p_n\| < (1-\epsilon)\log\lambda\right\}.
$$

\begin{thm}[Canonical strong stable manifold]\cite[Theorem 6.13]{CLPY1}\label{stable}
If $N = \infty$, then
$$
W^{ss}(p_0) := \bigcup_{n=0}^\infty F^{-n}(W^v_{\loc}(p_n)).
$$
Consequently, $W^{ss}(p_0)$ is a $C^{r+1}$-smooth manifold.
\end{thm}

If $M = \infty$, then \propref{hor angle shrink} implies that $E_{p_0}^h$ is the unique direction along which $p_0$ is infinitely backward regular. In this case, we denote $E_{p_0}^c := E_{p_0}^h$, and refer to this direction as the {\it center direction at $p_0$}. Moreover, we define the {\it (local) center manifold at $p_0$} as
$$
W^c(p_0) := \Phi_{p_0}^{-1}(\{(x, 0) \in U_{p_0}\}).
$$
Unlike strong stable manifolds, center manifolds are not canonically defined. However, the following result states that it still has a canonical jet.

\begin{thm}[Canonical jets of center manifolds]\cite[Theorem 6.16]{CLPY1}\label{center jet}
Suppose $M = \infty$. Let $\Gamma_0: (-t, t) \to \cU_{p_0}$ be a $C^{r+1}$-curve parameterized by its arclength such that $\Gamma_0(0) = p_0$, and for all $n \in \bbN$, we have
$$
\|DF^{-n}|_{\Gamma_0'(t)}\| < \lambda^{-\frac{(1-\bepsilon)n}{r+1}}
\matsp{for}
|t| < \lambda^{\epsilon n}.
$$
Then $\Gamma_0$ has a degree $r+1$ tangency with $W^c(p_0)$ at $p_0$.
\end{thm}

\subsection{Horizontal regularity}

We say that $p \in \cB$ is {\it $N$-times forward horizontally $(L, \epsilon, \lambda)$-regular along $E_p^{h,+} \in \bbP_p^2$} if, for $s\in \{1, 2\}$, we have
\begin{equation}\label{eq:hor for reg}
L^{-1}\lambda^{(1+\epsilon)n} \leq \frac{\Jac_p F^n}{\|D_pF^n|_{E_p^{h,+}}\|^s} \leq L\lambda^{(1-\epsilon)n}
\matsp{for}
1 \leq n \leq N.
\end{equation}
Similarly, we say that $p$ is {\it $M$-times backward horizontally $(L, \epsilon, \lambda)$-regular along $E_p^{h,-} \in \bbP_p^2$} if, for $s \in \{1, 2\}$, we have
\begin{equation}\label{eq:hor back reg}
L^{-1}\lambda^{(1+\epsilon)n}  \leq  \frac{\|D_pF^{-n}|_{E_p^{h,-}}\|^s}{\Jac_p F^{-n}} \leq L\lambda^{(1-\epsilon)n}
\matsp{for}
1 \leq n \leq M.
\end{equation}
If \eqref{eq:hor for reg} and \eqref{eq:hor back reg} hold with $E_p^h := E_p^{h,+} = E_p^{h,-}$, then $p$ is {\it $(M, N)$-times horizontally $(L, \epsilon, \lambda)$-regular along $E_p^h$}.

\begin{prop}[Horizontal vs vertical forward regularity]\cite[Proposition 5.2]{CLPY2}\label{transverse for reg}
If $p$ is $N$-times forward horizontally $(L, \epsilon, \lambda)$-regular along $E_p^h \in \bbP_p^2$, then there exists $E_p^v \in \bbP_p^2$ such that $p$ is $N$-times forward $(\bL, \bepsilon, \lambda)$-regular along $E_p^v$.
\end{prop}

\begin{prop}[Horizontal vs vertical backward regularity]\cite[Proposition 5.3]{CLPY2}\label{transverse back reg}
Suppose $p$ is $M$-times backward horizontally $(L, \epsilon, \lambda)$-regular along $E_p^h \in \bbP_p^2$. Let $E_p^v \in \bbP_p^2 \setminus \{E_p^h\}$. If $\measuredangle (E_p^h, E_p^v) > \theta$, then the point $p$ is $M$-times backward $(\bL/\theta^2, \bepsilon, \lambda)$-regular along $E_p^v$.
\end{prop}


\section{Distortion Theorems for 1D Maps}\label{sec.koebe}

In this section, we summarize some of the techniques in 1D dynamical systems used to control distortion. See \cite{dMvS} for complete details.

Let $f : I \to f(I)$ be a $C^1$-diffeomorphism on an interval $I \subset \bbR$. For $J \subset I$, the {\it distortion of $f$ on $J$} is defined as
$$
\Dis(f, J) := \sup_{x, y \in J}\frac{|f'(x)|}{|f'(y)|}.
$$
We denote $\Dis(f) := \Dis(f, I)$. For $K \geq 1$, we say that $f$ has {\it $K$-bounded distortion on $J$} if
$$
\Dis(f, J) \leq K.
$$
Clearly, if $g : I' \to g(I')$ is another $C^1$-diffeomorphism on an interval $I' \supset f(J)$, then we have
\begin{equation}\label{eq.dis product}
\Dis(g \circ f, J) \leq \Dis(g, f(J)) \cdot \Dis(f, J).
\end{equation}

\begin{thm}[Denjoy Lemma]\label{denjoy}
Let $f: I \to I$ be a $C^r$-map on an interval $I \subset \bbR$. Then there exists a uniform constant $K >0$ such that if $f^n|_J$ is a diffeomorphism on a subinterval $J \subset I$ for some $n \in \bbN$, then
$$
\log(\Dis(f^n, J)) \leq K \sum_{i=0}^{n-1}|f(J)|.
$$
\end{thm}

\subsection{Cross Ratios}

Let $J \Subset I \subset \bbR$ be bounded open intervals. The complement $I \setminus \overline{J}$ consists of two intervals $L$ and $R$. The {\it cross-ratio of $J$ in $I$} is given by
$$
\Cr(I, J) := \frac{|I||J|}{|L||R|}.
$$
For $\tau > 0$, we say that $I$ {\it contains a $\tau$-scaled neighborhood of $J$} if
$$
|L|, |R| > \tau |J|.
$$

Let $f : I \to f(I)$ be a homeomorphism. The {\it cross-ratio distortion under $f$ of $J$ in $I$} is given by
$$
\CrD(f, I, J) := \frac{\Cr(f(I), f(J))}{\Cr(I, J)}.
$$
Clearly, if $g : f(I) \to g\circ f(I)$ is another homeomorphism, then
\begin{equation}\label{eq.crd product}
\CrD(g\circ f, I, J)= \CrD(g, f(I), f(J)) \cdot \CrD(f, I, J).
\end{equation}
For $\nu > 0$, we say that $f$ has {\it $\nu$-bounded cross-ratio distortion on $I$} if
$$
\CrD(f, I', J) > \nu
$$
for all bounded open intervals $J \Subset I' \subset I$.

\begin{lem}\label{power crd}
For $\alpha > 1$, let $P_\alpha : \bbR^+ \to \bbR^+$ be an $\alpha$-power map such that
$$
P_\alpha(x) = x^\alpha
\matsp{for}
x \in \bbR^+.
$$
Then $P_\alpha|_{\bbR^+}$ has negative Schwarzian derivative. Consequently, $P_\alpha|_{\bbR^+}$ has $1$-bounded cross-ratio distortion on $\bbR_+$.
\end{lem}

\begin{lem}\label{d crd}
Let $I \subset \bbR$ be a bounded open interval, and let $f : I \to f(I)$ be a $C^1$-diffeomorphism with $K$-bounded distortion on $I$ for some $K >0$. Then there exists a uniform constant $\nu = \nu(K) > 0$ such that $f$ has $\nu$-bounded cross-ratio distortion on $I$.
\end{lem}

\begin{thm}[Koebe distortion theorem]\label{koebe}
Let $J \Subset I \subset \bbR$ be bounded open intervals, and let $f : I \to f(I)$ be a $C^1$-diffeomorphism with $\nu$-bounded cross-ratio distortion on $I$ for some $\nu >0$. If $f(I)$ contains a $\tau$-scaled neighborhood of $f(J)$, then there exists a uniform constant $K = K(\nu, \tau) >0$ depending only on $\nu$ and $\tau$ such that $f$ has $K$-bounded distortion on $J$.
\end{thm}


\section{Elementary $C^r$-Estimates}

\begin{lem}\cite[(4)]{PuSh}\label{hocb}
Let $F, G$ be $C^r$-maps such that $F \circ G$ is well-defined. Then
$$
\|F\circ G\|_r \leq r^r \|F\|_r \|G\|_r^r,
$$
where $\|F\|_r := \|DF\|_{C^{r-1}}$.
\end{lem}

\begin{lem}\cite[Lemma B.4]{CLPY1}\label{cr inverse}
Consider a $C^r$-diffeomorphism $f : \bbR \to \bbR$. Suppose $\|f'\| > \mu$ for some constant $\mu \in (0,1)$. Then there exists a uniform constant $C = C(r) \geq 1$ such that
$$
\|(f^{-1})'\|_{C^{r-1}} < C\mu^{1-2r}\|f''\|_{C^{r-2}}^{r-1}.
$$
\end{lem}

\begin{lem}\label{factor}
For $r \geq 4$, let $f : I \to f(I)$ be a $C^r$-map defined on an interval $0 \in I \subset \bbR$ such that $f(0) = 0 = f'(0)$ and $f''(0) = \kappa > 0$. Then there exists a $C^r$-diffeomorphism $\psi_f : I \to \psi_f(I)$ such that $f(x) = \kappa \cdot (\psi_f(x))^2$, and $\|\psi_f^{\pm 1}\|_{C^{r-3}} < C$ for some uniform constant $C = C(\|f\|_{C^r},\kappa, r) >0$.
\end{lem}

\begin{proof}
In the proof, let $K_i > 0$ for $i \in \bbN$ be uniform constants that depend only on $\|f\|_{C^r}$, $\kappa$ and $r$.

Write
$$
\kappa^{-1}f(x)-x^2 =  h(x) +\sum_{i=3}^r a_i x^i,
$$
where
$$
\lim_{x\to 0}\frac{h(x)}{x^r} = 0
\matsp{and}
\|h^{(r)}\| < K_1.
$$
Consequently,
\begin{equation}\label{eq.factor}
\|h^{(i)}\| < K_2 |x|^{r-i}
\matsp{for}
0\leq i \leq r.
\end{equation}

Define
$$
\psi_f(x) := x\sqrt{1+g(x)}
\matsp{where}
g(x) := \frac{\kappa^{-1}f(x)-x^2}{x^2}.
$$
Let $J := \{|x| < 1/K_3\}$. Observe that $f(x) > 1/K_4$ for $x \in I\setminus J$. Thus, applying \lemref{hocb}, we have
$$
\|(\psi_f|_{I\setminus J})^{\pm 1}\|_{C^r} < \left\|\left(\sqrt{f|_{I\setminus J}}\right)^{\pm 1}\right\|_{C^r} < K_5.
$$

Let $\hat h(x) := h(x)/x^2$. We claim that that $\hat h^{(k)}(x)$ with $k \leq r-3$ is a sum of a uniform number of terms of the form
\begin{equation}\label{eq.factor 1}
c\frac{h^{(i)}(x)}{x^{2+k-i}}
\end{equation}
for some coefficient $c \in \bbR$ independent on $f$ and $i \leq k$. Proceeding by induction, suppose that this is true for $k < r-3$. Differentiating, \eqref{eq.factor 1}, we obtain
$$
c\frac{h^{(i+1)}(x)}{x^{2+k-i}} + (2+k-i)c\frac{h^{(i)}(x)}{x^{2+k-i+1}}.
$$
The claim follows. Hence, by \eqref{eq.factor}, we conclude that
$$
|\hat h^{(k)}(x)|,\; |g^{(k)}(x)| < K_5|x|
\matsp{for}
0 \leq k \leq r-3.
$$
In particular, $\|g|_J\| \ll 1$.

A simple computation shows that $\|\psi_f|_J\|_{C^{r-3}} < K_6$, and $|\psi_f'(x)| > c$ for $x \in J$, where $c>0$ is an independent constant. Applying \lemref{cr inverse} to obtain the required bound for the inverse of $\psi_f$, the result follows.
\end{proof}


\bigskip

\begin{tabular}{l l l}
\emph{Sylvain Crovisier} &&
\emph{Mikhail Lyubich}\\

Laboratoire de Math\'ematiques d'Orsay &&
Institute for Mathematical Sciences\\

CNRS - Univ. Paris-Saclay &&
Stony Brook University\\

Orsay, France &&
Stony Brook, NY, USA\\

&&\\

\emph{Enrique Pujals} &&
\emph{Jonguk Yang}\\

Graduate Center - CUNY &&
Institut für Mathematik\\

New York, USA &&
Universität Zürich\\

&& Zürich, Switzerland

\end{tabular}

\end{document}